\documentclass[11pt]{amsart}
\usepackage{amscd,amssymb}
\usepackage{amsthm,amsmath}

\usepackage{fancyhdr}
\usepackage{supertabular}
\usepackage{setspace}
\usepackage{color}
\usepackage{longtable}

\usepackage[matrix,arrow,curve]{xy}
\usepackage{graphicx}
\usepackage{enumerate}
\usepackage{amsfonts}
\usepackage{cite}
\usepackage{pifont}
\usepackage{multirow}
\usepackage{hhline}
\usepackage[dvipsnames]{xcolor}

\usepackage{pgf,pgfarrows,pgfnodes,pgfautomata,pgfheaps,pgfshade}
\usepackage{tikz}

\makeatletter
\ifcase \@ptsize \relax
  \newcommand{\miniscule}{\@setfontsize\miniscule{3}{7}}
    \newcommand{\stiny}{\@setfontsize\miniscule{5}{7}}
\or
  \newcommand{\miniscule}{\@setfontsize\miniscule{3}{7}}
   \newcommand{\stiny}{\@setfontsize\miniscule{5}{7}}
\or
  \newcommand{\miniscule}{\@setfontsize\miniscule{3}{7}}
    \newcommand{\stiny}{\@setfontsize\miniscule{5}{7}}
\fi \makeatother

\newtheorem{theorem}[equation]{Theorem}
\newtheorem{lemma}[equation]{Lemma}
\newtheorem{proposition}[equation]{Proposition}
\newtheorem{corollary}[equation]{Corollary}
\newtheorem{conjecture}[equation]{Conjecture}

\newtheorem*{question*}{Question}

\newtheorem*{problem*}{Problem}

\theoremstyle{definition}
\newtheorem{example}[equation]{Example}
\newtheorem{definition}[equation]{Definition}
\newtheorem{remark}[equation]{Remark}

\theoremstyle{remark}

\makeatletter\@addtoreset{equation}{subsection} \makeatother

\newcommand{\mult}{\operatorname{mult}}

\title{Cylinders in del Pezzo surfaces}

\author{Ivan Cheltsov, Jihun Park and Joonyeong Won}

\address{ \emph{Ivan Cheltsov}\newline \textnormal{School of Mathematics, The
University of Edinburgh
\newline \medskip James Clerk Maxwell Building,
The King's Buildings,
Mayfield Road, Edinburgh EH9 3JZ, UK.\newline
Laboratory of Algebraic Geometry, National Research University Higher School of Economics
\newline 7 Vavilova Str. Moscow, 117312, Russia.\newline
 \texttt{I.Cheltsov@ed.ac.uk}}}

\address{ \emph{Jihun Park}\newline \textnormal{Center for Geometry and Physics, Institute for Basic Science (IBS)
\newline \medskip 77 Cheongam-ro, Nam-gu, Pohang, Gyeongbuk, 37673, Korea. \newline
Department of
Mathematics, POSTECH \newline
 77 Cheongam-ro, Nam-gu, Pohang, Gyeongbuk,  37673, Korea. \newline
\texttt{wlog@postech.ac.kr}}}

\address{\emph{Joonyeong
Won}\newline \textnormal{Algebraic Structure and its Applications Research Center, KAIST \newline 335 Gwahangno, Yuseong-gu, Daejeon, 34143, Korea\newline \texttt{leonwon@kias.re.kr}}}%

\begin{document}

\begin{abstract}
On del Pezzo surfaces, we study effective ample $\mathbb{R}$-divisors such that the complements of their supports are isomorphic to $\mathbb{A}^1$-bundles over smooth affine curves.
\end{abstract}

\maketitle

All considered varieties are assumed to be
algebraic and defined over an algebraically closed field of
characteristic $0$ throughout this article.

\section{Introduction}
\label{sec:intro}

\subsection{Cylinders}

The purpose of this article is to study cylinders in rational surfaces and, more specially, in del Pezzo surfaces.
A cylinder in a projective variety $X$ is a Zariski open subset that is isomorphic
to $\mathbb{A}^1\times Z$ for some affine variety $Z$.
So, if $X$ is a rational surface, then $Z$ is just the projective line with finitely many missing points. 

One can easily see that every smooth rational surface contains cylinders (\cite[Proposition~3.13]{KPZ11a}).
However, this is no longer true for singular rational surfaces,
i.e., there are plenty of singular rational surfaces without any cylinder.
Let us explain how to find such rational surfaces. 
First, let $S$ be a rational surface with quotient singularities
and  suppose that $S$ has a cylinder $U$, i.e., a Zariski open subset in $S$ such that $U\cong\mathbb{A}^1\times Z$ for some affine curve $Z$. Consider the following commutative diagram
\begin{equation}\label{diagram}
\xymatrix{\mathbb{P}^1\times\mathbb{P}^1\ar[dd]^{\bar{p}_{2}}&\mathbb{A}^1\times\mathbb{P}^1\ar@{_{(}->}[l]\ar[dd]^{p_{2}}&~\mathbb{A}^1\times Z\cong U\ar@{_{(}->}[l]\ar[d]^{p_Z}\ar@{^{(}->}[r] &S\ar@{-->}^{\psi}[dd]& &\tilde{S}\ar[ll]_{\pi}\ar[ddll]^{\phi}& \\
&&Z\ar@{^{(}->}[rd] \ar@{_{(}->}[dl]&&&& \\
\mathbb{P}^1\ar@{=}[r]&\mathbb{P}^1\ar@{=}[rr] &&\mathbb{P}^1 & &&}
\end{equation}
where $p_Z$, $p_{2}$, and $\bar{p}_2$ are the natural projections to the second factors, $\psi$ is the  rational map induced by $p_Z$, $\pi$ is a birational morphism resolving the indeterminacy of $\psi$, and $\phi$ is a morphism.
By construction, a general fiber of $\phi$ is $\mathbb{P}^1$.
Let $C_1,\ldots,C_n$ be the irreducible curves in $S$ such that
$$
S\setminus U=\bigcup_{i=1}^{n}C_i.
$$
Then the curves $C_1,\ldots,C_n$ generate the divisor class group $\mathrm{Cl}(S)$ of the surface $S$
because $\mathrm{Cl}(U)=0$. In particular, one has 
\begin{equation}\label{equation:KPZ-r}
n\geqslant\mathrm{rank}\,\mathrm{Cl}(S).
\end{equation}
Let $E_1,\ldots, E_r$ be the $\pi$-exceptional curves, if any,
and let $\Gamma$ be the section of $\bar{p}_2$ that is the complement of $\mathbb{A}^1\times\mathbb{P}^1$ in $\mathbb{P}^1\times\mathbb{P}^1$.
Denote by $\tilde{C}_1,\ldots,\tilde{C}_n$, and $\tilde{\Gamma}$ the proper transforms of the curves $C_1,\ldots,C_n$, and $\Gamma$
on the surface $\tilde{S}$, respectively. Then $\tilde{\Gamma}$ is a section of $\phi$. Moreover, the curve $\tilde{\Gamma}$
is one of the curves $\tilde{C}_1,\ldots,\tilde{C}_n$ and $E_1,\ldots, E_r$. Furthermore, all the other curves among
$\tilde{C}_1,\ldots,\tilde{C}_n$ and $E_1,\ldots, E_r$ are irreducible components of some fibers of $\phi$. 
We may assume either $\tilde{\Gamma}=\tilde{C}_1$ or $\tilde{\Gamma}=E_r$.

If $\tilde{\Gamma}=\tilde{C}_1$, then $\psi$ is a morphism. Conversely, if $\psi$ is a morphism, then $\tilde{\Gamma}=\tilde{C}_1$.

Let $\lambda_1,\ldots,\lambda_n$ be arbitrary real numbers. Then
$$
K_{\tilde{S}}+\sum_{i=1}^n\lambda_i\tilde{C}_i+\sum_{i=1}^r\mu_iE_i= \pi^*\Big(K_{S}+\sum_{i=1}^n\lambda_iC_i\Big)
$$
for some real numbers $\mu_1,\ldots,\mu_r$. Let $\tilde{F}$ be a general fiber of $\phi$.
Then $K_{\tilde{S}}\cdot\tilde{F}=-2$ by the adjunction formula.
Put $F=\pi(\tilde{F})$.

 If $\tilde{\Gamma}=E_r$, then
\begin{multline*}
-2+\mu_r
=\left(K_{\tilde{S}}+\sum_{i=1}^n\lambda_i\tilde{C}_i+\sum_{i=1}^r\mu_iE_i\right)\cdot\tilde{F}=\pi^*\left(K_{S}+\sum_{i=1}^n\lambda_iC_i\right)\cdot\tilde{F}=\left(K_{S}+\sum_{i=1}^n\lambda_iC_i\right)\cdot F
\end{multline*}

If $\tilde{\Gamma}=C_1$, then
\begin{multline*}
-2+\lambda_1
=\left(K_{\tilde{S}}+\sum_{i=1}^n\lambda_i\tilde{C}_i+\sum_{i=1}^r\mu_iE_i\right)\cdot\tilde{F}=\pi^*\left(K_{S}+\sum_{i=1}^n\lambda_iC_i\right)\cdot\tilde{F}=\left(K_{S}+\sum_{i=1}^n\lambda_iC_i\right)\cdot F.
\end{multline*}
On the other hand, if $K_S+\sum_{i=1}^n\lambda_iC_i$ is pseudo-effective, then
$$
\Big(K_{S}+\sum_{i=1}^n\lambda_iC_i\Big)\cdot F\geqslant 0
$$
because $\tilde{F}$ is a general fiber of $\phi$.  
\begin{remark}\label{observation}   
We are therefore able to draw the following conclusions:
\begin{itemize}
\item if $K_S+\sum_{i=1}^n\lambda_iC_i$ is pseudo-effective, the log pair $(S,\sum_{i=1}^n\lambda_iC_i)$ is not log canonical;

\item if $K_S+\sum_{i=1}^n\lambda_iC_i$ is pseudo-effective for some real numbers $\lambda_i <2$, the rational map
$\psi$ cannot be a morphism.

 \end{itemize}
 This observation is originated from  \cite[Lemma~4.11]{KPZ11a},
\cite[Lemma~4.14]{KPZ11a}, and \cite[Lemma~5.3]{KPZ12b}.
\end{remark}
 From Remark~\ref{observation} it immediately follows that if a rational surface with pseudo-effective canonical class has only quotient singularities then it cannot contain any cylinder. 
 We will present various examples of such surfaces in Section~\ref{section:example}.

\subsection{Polar cylinders}
In general, it seems hopeless to determine which singular rational surfaces have cylinders and which do not have cylinders.
This is simply because we do not have any reasonable classification of singular del Pezzo surfaces.
Instead of this, we are to consider a similar problem for polarized surfaces, which has a significant application to theory of unipotent group actions  in affine geometry (for instance, see \cite{KPZ11a}, \cite{KPZ12a}, \cite{KPZ12b}).
To do this, let $S$ be a rational surface with at most quotient singularities.

\begin{definition}[\cite{KPZ11a}]
\label{definition:polar-cylinder} Let $M$  be an
$\mathbb{R}$-divisor on $S$. An $M$-polar cylinder in $S$ is a
Zariski open subset $U$ of $S$ such that
\begin{itemize}
\item $U=\mathbb{A}^1\times Z$ for some affine curve $Z$, i.e., $U$ is a cylinder in $S$,
\item there is an 
effective $\mathbb{R}$-divisor $D$ on $S$ with $D\equiv
M$ and $U=S\setminus \mathrm{Supp}(D)$.
\end{itemize}
\end{definition}
With the notation at the beginning, the second condition can be rephrased as follows:
\begin{equation}\label{equation:num-equiv} 
M\equiv\sum_{i=1}^{n}\lambda_i C_i
\end{equation}
for some positive real numbers $\lambda_1,\ldots,\lambda_n$. We here remark  that
on a log del Pezzo surface, numerical equivalence for $\mathbb{Q}$-divisors coincides with $\mathbb{Q}$-linear equivalence, i.e.,
if $S$ is a log del Pezzo surface and $M$ is a $\mathbb{Q}$-divisor, then \eqref{equation:num-equiv} can be rewritten as
\[M\sim_{\mathbb{Q}}\sum_{i=1}^{n}\lambda_i C_i
\]
for some positive rational numbers $\lambda_1,\ldots,\lambda_n$.

Let $\mathrm{Amp}(S)$  be the ample cone of $S$ in $\mathrm{Pic}(S)\otimes\mathbb{R}$. Denote by $\mathrm{Amp}^{cyl}(S)$ the set
$$\left\{ H\in \mathrm{Amp}(S) : \mbox{ there is an $H$-polar cylinder on $S$}\right\}.$$ 
This set will be called the cone of cylindrical ample divisors of $S$.
We have seen  that the set $\mathrm{Amp}^{cyl}(S)$ can be empty.
On the other hand, one can show that $\mathrm{Amp}^{cyl}(S)\ne\varnothing$ provided that $S$ is smooth (see \cite[Proposition~3.13]{KPZ11a}).

For smooth del Pezzo surfaces, \cite{CheltsovParkWon} and \cite{KPZ12b}  have achieved the following:

\begin{theorem}
\label{theorem:smooth-del-Pezzos-degree-1-2-3}
Let $S_d$ be a  smooth del Pezzo surface of degree $d$. Then the set 
$\mathrm{Amp}^{cyl}(S_d)$  contains the anticanonical class if and only if $d\geqslant 4$.
\end{theorem}

Theorem~\ref{theorem:smooth-del-Pezzos-degree-1-2-3} has  been generalised in \cite{CheltsovParkWon2} as follows:

\begin{theorem}
\label{theorem:singular-del-Pezzos} Let $S_d$ be a del Pezzo surface of degree $d$ with  at most du Val singularities.
The set $\mathrm{Amp}^{cyl}(S_d)$ contains the anticanonical class  if and only if one of the following conditions holds:
\begin{itemize}
\item $d\geqslant 4$,

\item $d=3$ and $S_d$ is singular,

\item $d=2$ and $S_d$ has a singular point that is not of type $\mathrm{A}_1$,

\item $d=1$ and $S_d$ has a singular point that is not of type $\mathrm{A}_1$, $\mathrm{A}_2$, $\mathrm{A}_3$, or $\mathrm{D}_4$.
\end{itemize}
\end{theorem}

In \cite{CheltsovParkWon}  and \cite{CheltsovParkWon2} we have witnessed several vague pieces of  evidence  for the supposition that a cylinder polarized by an ample divisor can be obtained by manipulating an anticanonically polarized cylinder, if any, on a log del Pezzo surface.
\begin{conjecture}
\label{conjecture:Du-Val} A  log del Pezzo surface $S$ has a $(-K_{S})$-polar cylinder if
and only if $\mathrm{Amp}^{cyl}(S)=\mathrm{Amp}(S)$.
\end{conjecture}

\section{Main Results}

\subsection{Fujita invariant}
To investigate the cones of cylindrical ample divisors on log del Pezzo surfaces, we adopt the concept, so-called, the Fujita invariant of a log pair 
defined in \cite[Definition~2.2]{HTT}. This was
originally introduced by T.~Fujita, disguised as its negative value and  under the name Kodaira energy (\cite{F87}, \cite{F92}, \cite{F96},  \cite{F97} ).  This plays  essential roles in Manin's conjecture (see, for example, \cite{BM90}, \cite{HTT}).

Let $S$ be a log del Pezzo surface and $A$ be a big
$\mathbb{R}$-divisor on $S$.  It follows from Cone Theorem (see, for instance, \cite[Theorem~3.7]{KoMo}) that the Mori cone  $\overline{\mathbb{NE}}(S)$ of the surface $S$ is polyhedral.
\begin{definition}
For the log pair $(S, A)$, we define the Fujita invariant of~$(S,A)$ by
$$
\mu_A:=\mathrm{inf}\left\{\lambda\in\mathbb{R}_{>0}\ \Big|\ \text{the $\mathbb{R}$-divisor}\ K_{S}+\lambda A\ \text{is pseudo-effective}\right\}.%
$$
 The smallest extremal  face  $\Delta_{A}$ of the Mori cone $\overline{\mathbb{NE}}(S)$ that contains $K_{S}+\mu_A A$ is called the Fujita face of $A$. The Fujita rank of $(S,A)$  is defined by
$
r_A:=\dim \Delta_{A}.%
$
Note that  $r_A=0$ if and only if $-K_S\equiv \mu_AA$.
\end{definition}

\begin{remark}
In \cite[Definition~2.2]{HTT}, B.~Hassett, Sh.~Tanimoto, and Yu.~Tschinkel define the Fujita invariants  only for $\mathbb{Q}$-factorial varieties with canonical singularities.
For general varieties, they define the Fujita invariants by taking the pull-backs to smooth models (\cite{LTT}). 
\end{remark}

 Let $\phi_A\colon S\to
Z$ be the contraction given by the Fujita face $\Delta_{A}$ of the divisor $A$. Then either~$\phi_A$ is a
birational  morphism or a conic bundle with
$Z\cong\mathbb{P}^1$ (see, for instance, \cite[Subsection~8.2.6]{D12}). 
In the former case, the $\mathbb{R}$-divisor $A$ is said to be of type $B(r_A)$ and in the latter case
it is said to be of type $C(r_A)$.

Now we suppose that $S$ is 
 a smooth del Pezzo surface of degree $d\leqslant 7$. Then the Mori cone  $\overline{\mathbb{NE}}(S)$ of the surface $S$ is  generated by all the $(-1)$-curves in $S$ (see \cite[Theorem~8.2.23]{D12}).  Let $H$ be an ample $\mathbb{R}$-divisor on $S$.

If $H$ is of type $B(r_H)$, then  its Fujita face $\Delta_H$ is generated by $r_H$ 
disjoint $(-1)$-curves contracted by $\phi_H$, where  $r_H\leqslant 9-d$. If $H$ is of type $C(r_H)$, then
$\Delta_{H}$ is generated by the $(-1)$-curves in the $(8-d)$ reducible fibers of $\phi_H$. Each reducible fiber 
consists of two $(-1)$-curves that intersect transversally at a single
point.

Suppose that $H$ is
of type $B(r_H)$. Let $E_1,\ldots,E_{r_H} $   be all the $(-1)$-curves contained in
$\Delta_{H}$.  These are disjoint and generate the Fujita face $\Delta_{H}$. Therefore, 
\begin{equation}\label{remark:curves-in-face-birational}
K_{S}+\mu_H H\equiv\sum_{i=1}^{r_H}a_iE_i
\end{equation}
for some positive real numbers $a_1,\ldots,a_{r_H}$.
We have $a_i<1$ for every $i$ because  $H\cdot E_i>0$. Vice versa, for
every positive real numbers $b_1,\ldots,b_{r_H}<1$, the divisor
$$
-K_{S}+\sum_{i=1}^{r_H}b_iE_i
$$
is ample. The set of all ample $\mathbb{R}$-divisor classes of type $B(r_H)$ in $\mathrm{Pic}(S)\otimes \mathbb{R}$ 
is denoted by~$\mathrm{Amp}^{B}_{r_H}(S)$.

 Suppose that $H$ is
of type $C(r_H)$.  Note that $r_H=9-d$. There are a $0$-curve $B$ and $(8-d)$ disjoint $(-1)$-curves $E_1, E_2, E_3, \ldots, E_{8-d}$,  each of which is contained in a distinct fiber of $\phi_H$,
 such that
\begin{equation}\label{lemma:curves-in-face} 
K_{S}+\mu_H H\equiv aB+\sum_{i=1}^{8-d}a_iE_i
\end{equation}
for some positive real number $a$ and non-negative real numbers $a_1, \ldots, a_{8-d}<1$.
In particular, these curves generate the Fujita face $\Delta_{H}$.
Vice versa, for  every positive real number
$b$ and non-negative real numbers $b_1, \ldots, b_{8-d}<1$
 the divisor
$$
-K_{S}+bB+\sum_{i=1}^{8-d}b_iE_i
$$
is ample.

In the case of type $C(r_H)$, put $\ell_H=|\{a_i| a_i\ne 0\}|$. The $\mathbb{R}$-divisor $H$ is said to be of length
$\ell_H$.
The set of all ample $\mathbb{R}$-divisor classes of type $C(r_H)$  with length $\ell_H$ in $\mathrm{Pic}(S)\otimes \mathbb{R}$  is denoted by 
 $\mathrm{Amp}^{C}_{\ell_H}(S)$.
 It is clear that
 \[\mathrm{Amp}(S)=\bigcup_{\ell=0}^{8-d}\mathrm{Amp}^{C}_{\ell}(S)\cup
 \bigcup_{r=0}^{9-d}\mathrm{Amp}^{B}_{r}(S).\]
Note that
$\mathrm{Amp}^{B}_{0}(S)$ is the ray generated by the anticanonical class $[-K_S]$.

\subsection{Main Theorems}
The goal of the present article is to study the cones of cylindrical ample divisors of  smooth del Pezzo surfaces. This continues the work of T.~Kishimoto, Yu.~Prokhorov, and M.~Zaidenberg in \cite{KPZ12b}
and the work of I.~Cheltsov, J.~Park, and J.~Won in \cite{CheltsovParkWon} and \cite{CheltsovParkWon2}.

\begin{theorem}
\label{theorem:main-easy} Let $S_d$ be a smooth del Pezzo surface of degree $d$.
\begin{enumerate}
\item For $4\leqslant d\leqslant 9$, 
$$\mathrm{Amp}^{cyl}(S_d)=\mathrm{Amp}(S_d).$$

\item For $d=3$, 
$$\mathrm{Amp}^{cyl}(S_3)=\mathrm{Amp}(S_3)\setminus\mathrm{Amp}^{B}_{0}(S_3),$$
that is, any ample polarization  $H$ of $S_3$ admits an $H$-polar cylinder unless $H\equiv \alpha(-K_{S_3})$ for some $\alpha>0$. \end{enumerate}
\end{theorem}

A.~Perepechko verified that $\mathrm{Amp}^{cyl}(S_d)=\mathrm{Amp}(S_d)$ for $d\geqslant 5$ by showing 
that the ample cones of smooth del Pezzo surfaces of degrees at least $5$ are contained in the cones generated 
by components of a certain effective divisor the complement of the support of which is a cylinder (\cite[Subsection~3.2]{Pe13}).
However, his method cannot be fully generalised to del Pezzo surfaces of lower degrees. Indeed, he yielded
partial description of $\mathrm{Amp}^{cyl}(S_4)$ (\cite[Theorem~7]{Pe13}) and thereafter J.~Park and J.~Won showed 
that $\mathrm{Amp}^{cyl}(S_4)=\mathrm{Amp}(S_4)$ using the same idea as in the present article (\cite{PW16}).

The proof of Theorem~\ref{theorem:main-easy} is given in Subsection~\ref{section:big-degree}.

\begin{corollary}
\label{corollary-evidence} Conjecture~\ref{conjecture:Du-Val} holds for smooth del Pezzo surfaces.\end{corollary}

In order to analyze Conjecture~\ref{conjecture:Du-Val} for smooth del Pezzo surfaces we do not have to study 
the sets $\mathrm{Amp}^{cyl}(S)$ for del Pezzo surfaces of degrees $\leqslant 3$ since these surfaces do not have $(-K_S)$-polar cylinders already. Only Theorem~\ref{theorem:main-easy}~(1) is required. Meanwhile,
Theorem~\ref{theorem:main-easy}~(2) completely describes the set $\mathrm{Amp}^{cyl}(S)$ for smooth del Pezzo surfaces of degree 3.
After \cite{CheltsovParkWon} resolved the question whether  the anticanonical class lies in $\mathrm{Amp}^{cyl}(S)$ or not for a smooth cubic surface $S$,
Yu.~Prokhorov proposed a more general problem: 
\begin{center}
\emph{To describe the cones of cylindrical ample divisors of smooth cubic surfaces.}\\
\end{center}
Theorem~\ref{theorem:main-easy}~(2) gives the complete answer to Yu.~Prokhorov's problem. It is however natural that Yu.~Prokhorov's problem should be extended to smooth del Pezzo surfaces of degrees $\leqslant 2$. In the present article we also give some partial descriptions for them as follows:
\begin{theorem}
\label{theorem:main-non-existence} Let $S_d$ be a smooth del Pezzo surface
of degree $d\leqslant 3$ and $H$ be an ample $\mathbb{R}$-divisor of Fujita rank $r_H$ on $S_d$. 
If $r_H\leqslant 3-d$, then no $H$-polar cylinder exists on $S_d$.

\end{theorem}
Theorem~\ref{theorem:main-non-existence} immediately shows that
 \[\mathrm{Amp}^{cyl}(S_d)\bigcap\left\{
 \bigcup_{r=0}^{3-d}\mathrm{Amp}^{B}_{r}(S_d)\right\}=\emptyset\]
 for a smooth del Pezzo surface $S_d$
of degree $d\leqslant 3$.  The proof of the theorem follows from  Theorems~\ref{theorem:dp1-2-rank-1}~and~\ref{theorem:dp1-rank-2}.

\begin{theorem}
\label{theorem:main-hard} Let $S$ be a smooth del Pezzo surface
of degree $ 2$. Then
\begin{enumerate}

\item  \[\mathrm{Amp}^{cyl}(S)\supset 
 \bigcup_{r=3}^{7}\mathrm{Amp}^{B}_{r}(S).\]
 
  \item  \[\mathrm{Amp}^{cyl}(S)\bigcap\mathrm{Amp}^{B}_{2}(S)\ne\emptyset.\]

\item \[\mathrm{Amp}^{cyl}(S)\supset 
 \bigcup_{\ell=3}^{6}\mathrm{Amp}^{C}_{\ell}(S).\]

 \item  For each $0\leqslant \ell \leqslant 6$
 \[ \mathrm{Amp}^{cyl}(S)\bigcap\mathrm{Amp}^{C}_{\ell}(S)\ne\emptyset.\]

\end{enumerate}
\end{theorem}

In Theorem~\ref{theorem:main-hard}, (1) follows from Theorem~\ref{theorem:BR3-7}, (2) from Theorem~\ref{theorem:BR2},
(3) from Theorem~\ref{theorem:CL3-6}, and (4) from Theorem~\ref{theorem:CL0-2}.

\begin{theorem}
\label{theorem:main-hard-1} Let $S$ be a smooth del Pezzo surface
of degree $1$. Then
\begin{enumerate}
\item  For each $3\leqslant r \leqslant 8$
 \[\mathrm{Amp}^{cyl}(S)\bigcap\mathrm{Amp}^{B}_{r}(S)\ne\emptyset.\]
 \item  For each $0\leqslant \ell \leqslant 7$
 \[ \mathrm{Amp}^{cyl}(S)\bigcap\mathrm{Amp}^{C}_{\ell}(S)\ne\emptyset.\]
\end{enumerate}
\end{theorem}
The statement follows from Propositions~\ref{proposition:deg-1-cylinders-birational}~and~\ref{proposition:deg-1-cylinders-conic}.

\section{Rational singular  surfaces without any cylinder}
\label{section:example}
 
 \subsection{Examples}
Before we proceed, 
let us remind that  a normal surface singularity is Kawamata log terminal if and only if it is a quotient singularity (\cite[Corollary~1.9]{Ka84}). A projective surface with quotient singularities is always $\mathbb{Q}$-factorial. 

\begin{definition}  A normal projective surface  with quotient singularities is called  a log Enriques surface  if its canonical class is numerically trivial and its irregularity is zero.
It is called a log del Pezzo surface if its anticanonical class is  ample (\cite{MZh88}).
\end{definition}

We are now ready to present several examples of rational singular surfaces without any cylinder.

\begin{example}\label{example:Kollar}
J.~Koll\'ar has constructed a series of rational surfaces with ample canonical classes in \cite{Kollar}. The following is an easy example based on his construction.

Let $a_1, a_2, a_3, a_4$; $w_1, w_2,w_3, w_4$  be positive integers such that
\begin{itemize}
\item $a_1w_1+w_2=a_2w_2+w_3=a_3w_3+w_4=a_4w_4+w_1;$
\item  $\gcd ( w_1,w_3)=1$,  $\gcd ( w_2,w_4)=1$.
\end{itemize}
From the first condition above we obtain
\[w_1=(a_2a_3a_4-a_3a_4+a_4-1), \ \ \ w_2=(a_1a_3a_4-a_1a_4+a_1-1),\]
\[w_3=(a_1a_2a_4-a_1a_2+a_2-1), \ \ \ w_4=(a_1a_2a_3-a_2a_3+a_3-1).\]
Let $S$ be the Klein-type hypersurface in $\mathbb{P}(w_1,w_2,w_3,w_4)$ defined by the quasi-homogeneous equation of degree $(a_1a_2a_3a_4-1)$
\[x_1^{a_1}x_2+x_2^{a_2}x_3+x_3^{a_3}x_4+x_4^{a_4}x_1=0.\]
By the conditions $\gcd ( w_1,w_3)=1$ and  $\gcd ( w_2,w_4)=1$ we can easily see that $S$ is well-formed. Therefore, 
$$K_{S}=\mathcal{O}_S\left(a_1a_2a_3a_4-w_1-w_2-w_3-w_4-1\right).$$ 
 If $a_1$, $a_2$, $a_3$, $a_4\geqslant 4$, then $a_1a_2a_3a_4-w_1-w_2-w_3-w_4-1>0$, and hence 
$K_{S}$ 
is ample. 
By \cite[Theorem~39]{Kollar} $S$ is a rational surface of Picard rank three with four cyclic quotient singularities. 
Therefore, the  surface $S$ cannot contain any cylinder. In \cite{HwangKeum} D.~Hwang and J.~Keum have constructed another 
types of singular rational surfaces  with ample canonical divisors. 
\end{example}

\begin{example}
\label{example:Lee}
In order to construct smooth surface of general type with $p_g=q=0$ for a given self-intersection number of the canonical class, Y.~Lee, H.~Park, J.~Park, and D.~Shin generate rational elliptic surfaces with nef canonical classes  and quotient singularities of class T and then take their $\mathbb{Q}$-smoothings in \cite{LP07}, \cite{PPSh09a}, \cite{PPSh09b}.  Many such surfaces are presented in   \cite{LN13}. These rational surfaces meet our conditions not to have any cylinder. 
\end{example}

\begin{example}[ cf. \cite{OguisoTruong}]
\label{example:Ueno-Campana}
Let 
\[E=\mathbb{C}/(\mathbb{Z}+\tau\mathbb{Z})\]
be the elliptic curve of period $\tau=e^{\frac{2}{3}\pi}$.  Its $j$-invariant is $0$ and it is isomorphic to the Fermat cubic curve.
It is the unique elliptic curve admitting an automorphism $\sigma$ of order three such that $\sigma^* (\omega)=\tau \omega$,
where $\omega$ is a non-zero regular $1$-form on $E$. 
Let $S$ be the quotient surface
\[E\times E/ \langle\mathrm{diag}(-\sigma, -\sigma)\rangle.\]
Then, $6K_S$ is linearly trivial. Since there is no  non-zero regular $1$-form on $E\times E$ invariant by 
$\mathrm{diag}(-\sigma, -\sigma)$, we obtain $h^1(S, \mathcal{O}_S)=0$. Therefore, the surface $S$ is a rational surface with quotient singularities whose canonical class is numerically trivial, i.e., it is 
a log Enriques surface, and hence it cannot contain any cylinder.
\end{example}

\begin{example}
\label{example:log-Engiques} This construction is due to K.~Oguiso and D.-Q.~Zhang (\cite[Example~1]{OZh96}). The most extremal rational log Enriques surface of type $\mathrm{D}_{19}$ is defined as follows. Let $\overline{S}'$ be the quotient surface
 \[E\times E/ \langle\mathrm{diag}(\sigma, \sigma^2)\rangle,\]
 where the notations are the same as in Example~\ref{example:Ueno-Campana}. Note that the automorphism  $\sigma$ on $E$ has exactly  three fixed points $P_1$, $P_2$, $P_3$, respectively. They  correspond to $0$, $\frac{2}{3}+\frac{1}{3}\tau$, and $\frac{1}{3}+\frac{2}{3}\tau$.  The action by $ \langle\mathrm{diag}(\sigma, \sigma^2) \rangle$ on $E\times E$ has nine fixed points and these nine fixed points become du~Val singular points of type $\mathrm{A}_{2}$ on $\overline{S}'$. Let $S'$ be the minimal resolution of $\overline{S}'$. It is a K3~surface with twenty four smooth rational curves.
 Six of them come from the six fixed  curves, $\{P_i\}\times E$, $E\times\{P_i\}$ on $E\times E$. The others come from the nine singular points of type $\mathrm{A}_{2}$. Let $\sigma'$ be the automorphism of $S'$ induced by the automorphism  $\mathrm{diag}(\sigma, 1)$ on $E\times E$. Our twenty four smooth rational curves on $S'$ are $\sigma'$-invariant. Among these twenty four curves we can find a rational tree of  type $\mathrm{D}_{19}$.  Let $S'\to \hat{S}$ be the contraction of this tree. Then $\sigma'$ acts on $\hat{S}$ and it fixes two points. The quotient surface $\hat{S}/\langle \sigma'\rangle$ is a rational log Enriques surface.
 It follows from Remark~\ref{observation}   that any rational log Enriques surface cannot contain a cylinder. 
 
Rational log Enriques surfaces of ranks $19$ and $18$ are completely classified in \cite{OguisoZhang}, \cite{Wang}, respectively.
\end{example}

An effective $\mathbb{Q}$-divisor $D$ on a proper normal variety $X$ is called a tiger if it is numerically equivalent to $-K_X$ and $(X, D)$ is not log canonical. The tiger was introduced by S.~Keel and J.~Mckernan 
in  their study of log del Pezzo surfaces of Picard rank one (\cite{KeelMcKernan}).

\begin{example}\label{example:McKernan}
M.~Miyanishi proposed a conjecture (see \cite{GZh94}) that for a log del Pezzo surface of Picard rank one, its smooth locus has a 
finite unramified covering  that contains a cylinder.
It however turned out that the conjecture is answered in negative. S.~Keel and J.~Mckernan have constructed log del Pezzo surfaces of Picard rank one such that
\begin{itemize}
\item they have no tigers;
\item their smooth loci have trivial algebraic fundamental groups.
\end{itemize}
For their construction, see  \cite[Eample 21.3.3]{KeelMcKernan}. If such a surface $S$ contains a cylinder $U$, then we are able to obtain an effective divisor $D$ such that $S\setminus U=\mathrm{Supp}(D)$. Since $S$ is a log del Pezzo surface of Picard rank one, for some positive rational number $\lambda$ the divisor $\lambda D$ is linearly equivalent to $-K_S$. It immediately follows from Remark~\ref{observation}  that $\lambda D$ is a tiger. This is a contradiction. Therefore, the surface $S$ cannot contain any cylinder at all.
\end{example}

The rational surfaces in Exmple~\ref{example:McKernan} are overqualified to be free from cylinders. We may give away the condition of  the algebraic fundamental group since we are not considering cylinders in 
\'etale covers. On del Pezzo surfaces of Picard rank one with du Val singularities,  instead of non-existence of tigers, we can apply a finer condition than Remark~\ref{observation} that there is a tiger that does not contain the support of any effective 
anticanonical divisor (see \cite[Remark~3.8]{CheltsovParkWon2}).

\begin{example}
\label{example:dP1-D4-2A32A1-4A2} 
Let $S$ be a del Pezzo surface of degree $1$ with one of the following types of singularities
\[2\mathrm{D}_4, \  \ 2\mathrm{A}_3+2\mathrm{A}_1, \ \ 4\mathrm{A}_2.\]
Then its divisor class group  is generated by the anticanonical class over $\mathbb{R}$. Therefore, Theorem~\ref{theorem:singular-del-Pezzos}  implies that 
$S$ cannot contain any cylinder at all.
The smooth loci of these surfaces are not simply connected (\cite{MZh88}).
\end{example}

\begin{remark}
The surfaces in Example~\ref{example:dP1-D4-2A32A1-4A2}  are the only del Pezzo surfaces with du Val singularities that have no cylinder.
The other del Pezzo surfaces with du Val singularities contain cylinders. Indeed, del Pezzo surfaces with one of the types of singularities
$2\mathrm{D}_4$, $2\mathrm{A}_3+2\mathrm{A}_1$, $4\mathrm{A}_2$ are the only ones that contain no $(-K_{S})$-cylinder and have Picard rank one (see \cite{CheltsovParkWon2}). Furthermore, since their divisor class groups are generated by the anticanonical classes over $\mathbb{R}$,
they cannot contain any cylinder at all.
In \cite{CheltsovParkWon2}, all the del Pezzo surfaces with du Val singularities that have no anticanonically polarized cylinders are completely classified. Among them, the surfaces of the three types of singularities above are the only ones of Picard rank one.
For those of higher Picard rank without anticanonically polarized cylinders   one can always construct cylinders polarized by ample $\mathbb{Q}$-divisors. These constructions can be made by manipulating various anticanonically polarized cylinders that appear in \cite{CheltsovParkWon2}.
\end{remark}

\section{Cylinders in del Pezzo surfaces of big degree}

\subsection{Basic cylinders}
\label{section:cylinders}
We here present many examples of cylinders on smooth del Pezzo surfaces. These will be building blocks from which we are able to construct cylinders polarized by   various ample divisors.

Before we proceed, note that an $(n)$-curve on a smooth surface is an integral curve isomorphic to $\mathbb{P}^1$ with self-intersection number $n$.

\begin{example}\label{example:lines}
On $\mathbb{P}^2$, let $L_i$, $i=1,\ldots, r$, be lines meeting altogether  at a single point. The complement of the union of these $r$ lines 
 is an $\mathbb{A}^1$-bundle over  an $(r-1)$-punctured  affine line. Therefore, this is a cylinder in $\mathbb{P}^2$.
\end{example}
\begin{example}\label{example:line-and-conic}
Let $L$ and $M$ be a line and a conic on $\mathbb{P}^2$ intersecting tangentially at a point.
Each divisor $aL+bM$ with positive real numbers $a$ and $b$ defines a cylinder isomorphic to 
an $\mathbb{A}^1$-bundle over a simply punctured affine line $\mathbb{A}^1_*$.
\end{example}
\begin{example}\label{example:cusp}
Let $C$ be a cuspidal cubic curve on $\mathbb{P}^2$ and $T$ be the Zariski tangent line at its cuspidal point $P$. Their complement is a cylinder isomorphic to $\mathbb{A}^1\times \mathbb{A}^1_*$. To see this, take the blow up $\pi_1:\mathbb{F}_1\to \mathbb{P}^2$ at the point $P$ and then take the blow up $\pi_2:\tilde{S}_7\to \mathbb{F}_1$ at the intersection point of the proper transforms of $C$ and $T$. Let $E_1$ be the exceptional curve on $\tilde{S}_7$ contracted by $\pi_1$ and $E_2$ be the exceptional curve on $\tilde{S}_7$ contracted by $\pi_2$. In addition, let $\tilde{C}$ and $\tilde{T}$ be the proper transforms of $C$ and $T$ by $\pi_1\circ\pi_2$. 
Now we contract $\tilde{T}$, which is a $(-1)$-curve on $\tilde{S}_7$. This gives us a birational morphism $\pi:\tilde{S}_7\to \mathbb{F}_2$.  Then $\pi(E_1)$ is the exceptional section of $\mathbb{F}_2$  with  self-intersection number $-2$ and $\pi(E_2)$ is a fiber of  $\mathbb{F}_2$. The curve  $\pi(\tilde{C})$ is linearly equivalent to $\pi (E_1)+3\pi(E_2)$ and is a section 
of $\mathbb{F}_2$ with the self-intersection number $4$. The three curves 
$\pi(E_1)$, $\pi(E_2)$, and $\pi(\tilde{C})$ on $\mathbb{F}_2$ meet transversally at a single point, and hence  their complement is isomorphic to $\mathbb{A}^1\times\mathbb{A}^1_*$.  Since
\[\mathbb{P}^2\setminus(C\cup T)\cong \tilde{S}_7\setminus(\tilde{C}\cup \tilde{T}\cup E_1\cup E_2)
\cong \mathbb{F}_2\setminus(\pi(\tilde{C}) \cup \pi(E_1)\cup \pi(E_2))\]
the complement of $C$ and $T$ is a cylinder.
\end{example}

\begin{example}\label{example:P1-bundle}
On $\mathbb{P}^1\times\mathbb{P}^1$, let $E$ be a curve of bidegree $(1,0)$ and $F_i$, $i=1,\ldots, r$, be curves of bidegree $(0,1)$.  The complement of these curves is isomorphic to an $\mathbb{A}^1$-bundle over  a $(r-1)$-punctured affine line.
By blowing up $\mathbb{P}^1\times\mathbb{P}^1$ at the intersection point of $E$ and $F_1$ and then contracting the proper transforms of $E$ and $F_1$ to $\mathbb{P}^2$, we may encounter the cylinder described in Example~\ref{example:lines}.

\end{example}

\begin{example}[{\cite[Theorem~3.19]{KPZ11a}}]
\label{example:KPZ-actions-4-9} 
Let $S_d$ be a smooth del Pezzo surface of degree $d\geqslant 2$, not isomorphic to $\mathbb{P}^1\times\mathbb{P}^1$.
Then it can be obtained by blowing up $\mathbb{P}^2$ at $(9-d)$ points  in general position. Let $\sigma: S_d\to \mathbb{P}^2$ be such a blow up and let $E_i$ be the exceptional curve, where $i=1,\ldots, 9-d$. Put $P_i=\sigma(E_i)$.
Suppose that these $(9-d)$ points $P_1,\ldots, P_{9-d}$ lie on the union of a line $L$ and a conic $C$ that intersect tangentially at a single point.  Note
$$S_d\setminus (\tilde{L} \cup \tilde{C}\cup E_1\cup \ldots\cup E_{9-d})\cong \mathbb{P}^2\setminus (L\cup C)
\cong \mathbb{A}^1\times\mathbb{A}^1_*,$$
where $\tilde{L}$ and $\tilde{C}$ are the proper transforms of $L$ and $C$ by $\sigma$.

We now suppose that $d\geqslant 4$. There is a conic passing through all the points  $P_1,\ldots, P_{9-d}$. So we may assume that 
$P_1,\ldots, P_{9-d}$ lie on $C$ and  that $L$ meets $C$ tangentially at a point other than  $P_1,\ldots, P_{9-d}$.
Then, for a real number
$0<\varepsilon<\frac{1}{2}$, the $\mathbb{R}$-divisor 
$$(1+\varepsilon)\tilde{C}+(1-2\varepsilon)\tilde{L}+\varepsilon\sum_{i=1}^{9-d} E_i$$ 
is effective and numerically equivalent to $-K_{S_d}$.
Therefore, the cylinder
$$S_d\setminus (\tilde{L} \cup \tilde{C}\cup E_1\cup \ldots\cup E_{9-d})$$ is $(-K_{S_d})$-polar.
\end{example}

\begin{example}\label{example:non-Prokhorov example8-1}
Let $C$ be an irreducible curve of bidegree $(1,2)$ on $\mathbb{P}^1\times \mathbb{P}^1$. 
There are two curves of bidegree $(1,0)$ that intersect $C$ tangentially. Let $T$ be one of them and $P$ be the intersection point of $C$ and $T$.  In addition, let $L$ be the curve of bidegree $(0,1)$ passing through the point $P$.
Then the open set 
\[\mathbb{P}^1\times\mathbb{P}^1\setminus (C \cup L\cup T)\]
is a cylinder. To see this, let $\pi:S_7\to \mathbb{P}^1\times\mathbb{P}^1$ be the blow up at $P$ and $E$ be its exceptional curve. Denote the proper transforms of the curves $C$, $T$, and $L$ by $\tilde{C}$, $\tilde{T}$, and $\tilde{L}$, respectively.
The curves  $\tilde{T}$ and $\tilde{L}$ are disjoint $(-1)$-curves.
 Let $\psi:S_7\to \mathbb{P}^2$ be the contraction of $\tilde{L}$ and $\tilde{T}$. Then, $\psi(E)$ is a line, $\psi(C)$ is a conic, and they meet tangentially. Therefore, the open set above is a cylinder since
 \[\mathbb{P}^1\times\mathbb{P}^1\setminus (C \cup L\cup T)\cong 
 S\setminus (\tilde{C} \cup \tilde{L}\cup \tilde{T}\cup E)\cong 
 \mathbb{P}^2\setminus (\psi(\tilde{C})\cup \psi(E))\cong \mathbb{A}^1\times\mathbb{A}^1_*.\]
 In particular, for a real number $0<\varepsilon<1$,
$$(1-\varepsilon)C+(1+\varepsilon )T+ 2\varepsilon L\equiv -K_{\mathbb{P}^1\times\mathbb{P}^1},$$ and hence
the cylinder is  $(-K_{\mathbb{P}^1\times\mathbb{P}^1})$-polar.
\end{example}

\begin{example}\label{example:non-Prokhorov example8-0}
Let $C_1$ and $C_2$ be  irreducible curves of bidegree $(1,1)$ on $\mathbb{P}^1\times \mathbb{P}^1$.
Suppose that these two curves meet tangentially at a single point $P$. Let $L_1$ and $L_2$ be the curves of bidegrees $(1,0)$ and $(0,1)$, respectively, that pass through the point $P$.
We claim that the open set 
\[\mathbb{P}^1\times\mathbb{P}^1\setminus (C_1\cup C_2 \cup L_1\cup L_2)\]
is a cylinder.   Let $\pi:S_7\to\mathbb{P}^1\times\mathbb{P}^1$ be the blow up at $P$ and $E$ be its exceptional curve. Denote the proper transforms of $C_i$ and $L_i$ by $\tilde{C}_i$ and $\tilde{L}_i$, respectively, $i=1,2$. The curves  $\tilde{L}_1$ and $\tilde{L}_2$ are disjoint $(-1)$-curves. Therefore, we  obtain a birational morphism $\psi: S_7\to \mathbb{P}^2$ by contracting $\tilde{L}_1$ and $\tilde{L}_2$. The three curves
$\psi(E)$, 
 $\psi(\tilde{C}_1)$, $\psi(\tilde{C}_2)$ are lines intersecting at a single point. Since
 \[\mathbb{P}^1\times\mathbb{P}^1\setminus (C_1\cup C_2 \cup L_1\cup L_2)\cong 
 S_7\setminus (\tilde{C}_1\cup \tilde{C}_2 \cup \tilde{L}_1\cup \tilde{L}_2\cup E)\cong 
 \mathbb{P}^2\setminus (\psi(\tilde{C}_1)\cup \psi(\tilde{C}_2)\cup\psi(E))\cong \mathbb{A}^1\times\mathbb{A}^1_{**},\]
 where $\mathbb{A}^1_{**}$ is a 2-punctured affine line,
 our claim is confirmed.
 In particular, the cylinder is $(-K_{\mathbb{P}^1\times\mathbb{P}^1})$-polar because 
 the divisor $$(1+\varepsilon)C_1+(1-2\varepsilon )C_2+ \varepsilon L_1+ \varepsilon L_2$$ is effective 
for an arbitrary real number $0<\varepsilon<\frac{1}{2}$ and nemerically equivalent to $-K_{\mathbb{P}^1\times\mathbb{P}^1}$. 
 \end{example}

\begin{example}\label{example:non-Prokhorov example8-2}
Let $C$ be a cuspidal rational curve in the anticanonical linear system of the Hirzebruch surface $\mathbb{F}_1$. Let $M$ be the $0$-curve that passes through the cuspidal point $P$. There is a unique $1$-curve $T$ that intersects $C$ only at the point $P$. For an arbitrary real number $0<\varepsilon<1$, the divisor $$(1-\varepsilon)C+\varepsilon M+ 2\varepsilon T$$ defines a $(-K_{\mathbb{F}_1})$-polar cylinder. To see this, take the blow up $\phi:S_7\to \mathbb{F}_1$ at the point $P$. Let $E$ be the exceptional curve. Denote the proper transforms of the curves $C$, $M$, and $T$  on $S_7$ by $\tilde{C}$, $\tilde{M}$, and $\tilde{T}$, respectively. The $\mathbb{R}$-divisor 
$(1-\varepsilon)\tilde{C}+\varepsilon \tilde{M}+ 2\varepsilon \tilde{T}+ (1+\varepsilon)E$ is numerically equivalent to $-K_{S_7}$.
Let $\psi :S_7\to \mathbb{P}^1\times \mathbb{P}^1$ be the contraction of the $(-1)$-curve $\tilde{M}$. Then the $\mathbb{R}$-divisor 
$(1-\varepsilon)\psi(\tilde{C})+ 2\varepsilon \psi(\tilde{T})+ (1+\varepsilon)\psi(E)$  is the divisor in 
Example~\ref{example:non-Prokhorov example8-1} that defines a $(-K_{\mathbb{P}^1\times\mathbb{P}^1})$-polar cylinder.
\end{example}

\begin{example}\label{example:non-Prokhorov example8-3}
On the Hirzebruch surface $\mathbb{F}_1$, let $M$ be a $0$-curve and let $T$ be a $1$-curve. Let $P$ be the intersection  point of $T$ and $M$. There is a $3$-curve $C$ that is tangent to the curve $T$ at the point $P$.
Then for  an arbitrary real number $0<\varepsilon<1$, the divisor $$(1-\varepsilon)C+ (1+\varepsilon)T+
\varepsilon M$$ defines a $(-K_{\mathbb{F}_1})$-polar cylinder.  Take the blow up $\phi:S_7\to \mathbb{F}_1$ at the point $P$. Let $E$ be its exceptional curve. Denote the proper transforms of the curves $C$, $M$, and $T$  on $S_7$ by $\tilde{C}$, $\tilde{M}$, and $\tilde{T}$, respectively. The $\mathbb{R}$-divisor 
$(1-\varepsilon)\tilde{C}+ (1+\varepsilon) \tilde{T}+\varepsilon \tilde{M}+ (1+\varepsilon)E$ is numerically equivalent to $-K_{S_7}$.
Let $\psi :S_7\to \mathbb{P}^1\times \mathbb{P}^1$ be the contraction of the $(-1)$-curve $\tilde{M}$. Then the $\mathbb{R}$-divisor 
$(1-\varepsilon)\psi(\tilde{C})+ (1+\varepsilon )\psi(\tilde{T})+ (1+\varepsilon)\psi(E)$  is 
 numerically  equivalent to $-K_{\mathbb{P}^1\times\mathbb{P}^1}$. The curve $\psi(\tilde{C})$ is an irreducible curve of bidegree $(1,1)$, $\psi(\tilde{T})$ is of bidegree $(1,0)$, and $\psi(\tilde{E})$ is of bidegree $(0,1)$.
Moreover, these three curves meet at a single point.  This easily shows that the divisor $(1-\varepsilon)C+ (1+\varepsilon)T+
\varepsilon M$ defines a $(-K_{\mathbb{F}_1})$-polar cylinder. 
\end{example}

\begin{example}\label{example:non-Prokhorov example7-1}
Let $C$ be a cuspidal rational curve in the anticanonical linear system of the smooth del Pezzo surface  $S_7$ of degree 7. 
There are exactly two $0$-curves $M_1$, $M_2$ passing through the cuspidal point $P$ of $C$.
There is a unique $1$-curve $T$ that meets $C$ only at the point $P$. For an arbitrary real number $0<\varepsilon<1$, the divisor $$(1-\varepsilon)C+\varepsilon M_1+\varepsilon M_2+ \varepsilon T$$ defines a $(-K_{S_7})$-polar cylinder. 
To see this, take the blow up $\phi:S_6\to S_7$ at the point $P$. Let $E$ be the exceptional divisor. 
The proper transforms of $M_1$ and $M_2$ are disjoint $(-1)$-curves on $S_6$. Let $\psi: S_6\to \mathbb{F}_1$ be the birational morphism obtained by contracting these two curves. Then the curve $E$ becomes a $1$-curve and the curve $T$ becomes a $0$-curve on $\mathbb{F}_1$. They intersect at a single point $Q$. The curve $C$ becomes a $3$-curve tangent to $\psi(E)$ at $Q$.
Then
Example~\ref{example:non-Prokhorov example8-3}  shows that the divisor $(1-\varepsilon)C+\varepsilon M_1+\varepsilon M_2+ \varepsilon T$ defines a $(-K_{S_7})$-polar cylinder.
 \end{example}

\begin{example}\label{example:non-Prokhorov example6-1}
Let $S_6$ be a smooth del Pezzo surface of degree $6$.  Let $C$ be a cuspidal rational curve in the anticanonical linear system and let $P$ be its cuspidal point. There are three disjoint $(-1)$-curves $E_1, E_2, E_3$ on $S_6$. Each of them meets $C$ at a single smooth point of $C$.
Let $\phi:S_6\to \mathbb{P}^2$ be the birational morphism obtained by contracting $E_1,E_2, E_3$.
Let $T'$ be the Zariski tangent line to the cuspidal rational curve $\phi(C)$ at the point $\phi(P)$. There is a unique conic curve $T_0'$ such that it is tangent to $T'$ at the point $\phi(P)$ and  passes through the points $\phi(E_1), \phi(E_2), \phi(E_3)$.  Let $M_i'$ be the line through $\phi(P)$ and $\phi(E_i)$.
Let $T$, $T_0$, and $M_i$ be the proper transforms of $T'$, $T_0'$, and $M_i$ by the birational morphism~$\phi$.
For an arbitrary real number $0<\varepsilon<\frac{1}{2}$, the  $\mathbb{R}$-divisor 
$$(1-2\varepsilon)C+\varepsilon T+\varepsilon T_0+\varepsilon M_1+\varepsilon M_2+ \varepsilon M_3
$$ 
is numerically equivalent to the anticanonical class on $S_6$. 
We take the blow up $\phi:S_5\to S_6$ at the point $P$. Let $E$ be the exceptional curve. 
The proper transforms of $M_1$, $M_2$, and $M_3$ are disjoint $(-1)$-curves on $S_5$. 
Let $\psi: S_5\to \mathbb{P}^1\times\mathbb{P}^1$ be the birational morphism obtained by contracting these three curves. Then
Example~\ref{example:non-Prokhorov example8-0}  shows that the $\mathbb{R}$-divisor above  defines a $(-K_{S_6})$-polar cylinder on $S_6$.

 \end{example}

\begin{example}\label{example:non-Prokhorov example5-2}
Let $S_5$ be a smooth del Pezzo surface of degree $5$. In addition, let $H$ be an effective anticanonical divisor on $S_5$ that consists of one $1$-curve $C$ and one $0$-curve $M$ meeting tangentially at a single point $P$. 
Then there are four $0$-curves $M_1$, $M_2$, $M_3$, $M_4$, other than $M$, passing through the point $P$. They intersect  each other only at the point $P$. For an arbitrary real number $0<\varepsilon<\frac{1}{2}$, the divisor 
$$(1-\varepsilon)M+(1-2\varepsilon)C+\varepsilon M_1+\varepsilon M_2+ \varepsilon M_3+ 
\varepsilon M_4$$ 
defines a $(-K_{S_5})$-polar cylinder. Indeed, to see this, consider  the blow up $\phi:S_4\to S_5$ at the point $P$, and then contract the proper transforms of $M_i$'s and $M$ by $\phi$ to $\mathbb{P}^2$.

 \end{example}

\begin{example}\label{example:non-Prokhorov example5-4}
Let $S_5$ be a smooth del Pezzo surface of degree $5$.  Let $C$ be a cuspidal rational curve in the anticanonical linear system and let $P$ be its cuspidal point. We have four disjoint $(-1)$-curves $E_1,\ldots, E_4$ on $S_5$. Each of them intersects $C$ at a single smooth point of $C$.
Let $\phi:S_5\to \mathbb{P}^2$ be the birational morphism obtained by contracting $E_1,\ldots, E_4$.
Let $M'_0$ be the conic on $\mathbb{P}^2$ determined by  the points $\phi(E_1), \ldots, \phi(E_4)$, and $\phi(P)$. For $i=1,2,3,4$, let $M_i'$ be the line passing through  the points $\phi(E_i)$ and $\phi(P)$.  Let $M_i$ be the proper transform of $M_i'$ by the birational morphism $\phi$. These five curves are $0$-curves on $S_5$ passing through $P$.
For an arbitrary real number $0<\varepsilon<\frac{1}{2}$, the divisor 
$$(1-2\varepsilon)C+\varepsilon M_0+\varepsilon M_1+\varepsilon M_2+ \varepsilon M_3+ 
\varepsilon M_4$$ 
defines a $(-K_{S_5})$-polar cylinder. We consider the blow up $\psi: S_4\to S_5$ at $P$. The proper transforms of $M_i$ by $\psi$ are mutually disjoint $(-1)$-curves. We contract these five $(-1)$-curves to $\mathbb{P}^2$. Then $C$  and  the exceptional curve of $\psi$ become a line and a conic meeting tangentially on $\mathbb{P}^2$. Therefore, the effective  $\mathbb{R}$-divisor above defines a cylinder.
\end{example}

\begin{example}\label{example:non-Prokhorov example5-3} We here use the same  notations $S_5$, $C$,  $P$, $E_1,\ldots, E_4$, and $\phi:S_5\to \mathbb{P}^2$  as in Example~\ref{example:non-Prokhorov example5-4}.
Let $T'_0$ be the Zariski tangent line to the cuspidal rational curve $\phi(C)$ at the point $\phi(P)$. For $i=1,2,3,4$, there is a unique conic curve $T_i'$ such that it is tangent to $T'_0$ at the point $\phi(P)$ and passes through the points $\phi(E_1), \ldots, \phi(E_4)$ except $\phi(E_i)$. Let $T_0$ and $T_i$ be the proper transforms of $T'_0$ and $T_i'$ by the birational morphism $\phi$.
Then, for an arbitrary real number $0<\varepsilon<\frac{1}{3}$, the divisor 
$$(1-3\varepsilon)C+\varepsilon T_0+\varepsilon T_1+\varepsilon T_2+ \varepsilon T_3+ 
\varepsilon T_4$$ 
defines a $(-K_{S_5})$-polar cylinder. 
To see this, let $\phi_1:S_4\to S_5$ be the blow up at the point $P$. Then the exceptional curve $B_1$ of $\phi_1$ intersects the proper transform of $C$ tangentially at a single point $Q$. Let $\phi_2:S_3\to S_4$ be the blow up at the point $Q$ and denote its exceptional curve by $B_2$. Let $\tilde{C}$ and $\tilde{T}_i$ be the proper transforms of $C$ and $T_i$ by the morphism $\phi_1\circ\phi_2$, where $i=0,\ldots, 4$. In addition, let $\tilde{B}_1$ be the proper transform of $B_1$ by $\phi_2$. Then the  $\mathbb{R}$-divisor 
$$(1-3\varepsilon)\tilde{C}+(1-\varepsilon)\tilde{B}_1+(1+\varepsilon)B_2+\varepsilon \tilde{T}_0+\varepsilon  \tilde{T}_1+\varepsilon  \tilde{T}_2+ \varepsilon  \tilde{T}_3+ 
\varepsilon  \tilde{T}_4$$
is numerically equivalent to the anticanonical class on $S_3$. Now, contracting the five disjoint $(-1)$-curves $\tilde{T}_0,\ldots, \tilde{T}_4$ to the Hirzebruch surface $\mathbb{F}_2$, we obtain a birational morphism $\psi_1:S_3\to \mathbb{F}_2$. Put $\overline{C}=\psi_1(\tilde{C})$, 
$\overline{B}_1=\psi_1(\tilde{B}_1)$, and $\overline{B}_2=\psi_1(B_2)$. Then
$$(1-3\varepsilon)\overline{C}+(1-\varepsilon)\overline{B}_1+(1+\varepsilon)\overline{B}_2$$
is numerically equivalent to $-K_{\mathbb{F}_2}$. Note that $\overline{B}_1$ is the $(-2)$-curve on  $\mathbb{F}_2$ and $\overline{C}$ is a $0$-curve. We take the blow up $\phi_3:S\to \mathbb{F}_2$ at a general point of $\overline{C}$. Let $B_3$ be its exceptional curve. In addition,
let $\hat{C}$, 
$\hat{B}_1$, and $\hat{B}_2$ be the proper transforms of  $\overline{C}$, 
$\overline{B}_1$, and $\overline{B}_2$ by $\phi_3$. Finally, by contracting $\hat{C}$ and $\hat{B}_1$ in turn, we obtain a birational morphism $\psi_2:S\to\mathbb{P}^2$.  We immediately  see  that $\psi_3(\hat{B}_2)$ is a cuspidal rational curve and  $\psi_3(B_3)$ is the Zariski tangent line to
$\psi_3(\hat{B}_2)$ at its cuspidal point. Even though 
$$(1+\varepsilon)\psi_3(\hat{B}_2)-3\varepsilon\psi_3(B_3)$$
is a non-effective divisor on $\mathbb{P}^2$, the original divisor on $S_5$ defines a $(-K_{S_5})$-polar cylinder since
 \[S_5\setminus (C \cup T_0\cup\ldots\cup T_4)\cong 
 \mathbb{P}^2\setminus (\psi_3(\hat{B}_2)\cup \psi_3(B_3))\]
 (see Example~\ref{example:cusp}).
 \end{example}

\begin{example}\label{example:non-Prokhorov example4}
Let $S_4$ be a smooth del Pezzo surface of degree $4$. In addition, let $H$ be an effective anticanonical divisor on $S_4$ that consists of one $1$-curve $C$ and one $(-1)$-curve $L$ meeting tangentially at a single point $P$.

By contracting $L$, we see from Example~\ref{example:non-Prokhorov example5-3} that 
there are five $0$-curves $T_1, \ldots, T_5$ passing through the point $P$ such that 
they intersect each other only at $P$ and meet $L$ and $C$ only at~$P$. 
Example~\ref{example:non-Prokhorov example5-3} immediately shows that 
\[(1-3\varepsilon)C+\ (1-\varepsilon)L+\varepsilon\sum_{i=1}^5T_i\]
defines a $(-K_{S_4})$-polar cylinder on $S_4$ for $0<\varepsilon<\frac{1}{3}$.
 \end{example}

\subsection{Proof of Theorem~\ref{theorem:main-easy}}
\label{section:big-degree}

Let
$S_d$ be a smooth del Pezzo surface of degree $d\geqslant 3$ and let $H$ be an ample
$\mathbb{R}$-divisor on $S_d$. 
It is easy to check that $S_d$ always 
contains an $H$-polar cylinder isomorphic to $\mathbb{A}^2$ if $d\geqslant 8$. For this reason, in order to prove Theorem~\ref{theorem:main-easy}, we  assume that
$d\leqslant 7$.

Let $\mu$ be the Fujita invariant of $H$,  $\Delta$ be the Fujita face of $H$, and $r$ be the Fujita rank of $H$.
Let $\phi: S_d\to Z$ be the contraction given by~$\Delta$.

\begin{lemma}[{cf. Example~\ref{example:KPZ-actions-4-9}}]
\label{lemma:deg-3-cylinders-birational-1} Suppose that
$H$ is of type $B(r)$ and
$Z\not\cong\mathbb{P}^1\times\mathbb{P}^1$. If $(d, r)\ne (3,0)$, then $S_d$ contains an
$H$-polar cylinder.
\end{lemma}

\begin{proof} For the case $r=0$ and $d\geqslant 4$, Theorem~\ref{theorem:smooth-del-Pezzos-degree-1-2-3} implies the statement. Therefore, we may assume that $r>0$.

Let $E_1,\ldots,E_r$ be the $(-1)$-curves that generate the face $\Delta$. Then
$$
K_{S_d}+\mu H\equiv\sum_{i=1}^ra_iE_i
$$
for some positive real numbers $a_1,\ldots,a_r$. Note that 
$r\leqslant 9-d$ and $E_1,\ldots,E_r$ are disjoint.

The surface $Z$ is a smooth del Pezzo surface of degree $(d+r)$. Since 
$Z\not\cong\mathbb{P}^1\times\mathbb{P}^1$, either
$Z=\mathbb{P}^2$ or $Z$ is a blow up of $\mathbb{P}^2$ in
$(9-d-r)$  points in general position. Let $\psi\colon
Z\to\mathbb{P}^2$ be the blow up.
 Put
$k=9-d$ and $\sigma=\psi\circ\phi$. If $k>r$, denote the
proper transforms of these $\psi$-exceptional curves on $S_d$ by
$E_{r+1},\ldots,E_k$. Put $P_i=\sigma(E_i)$.

Let $C$ be an irreducible conic in $\mathbb{P}^2$ passing through the points $P_2,\ldots,P_{k}$. Such a conic exists because $k\leqslant 6$.
Let $L$ be a line in $\mathbb{P}^2$ passing through the point $P_1$ and tangent to the conic $C$. Note that $L$ may be tangent to $C$ at one of the points $P_2,\ldots, P_k$, say $P_2$.

For a positive real number $\varepsilon$ we have
$-K_{\mathbb{P}^2}\equiv (1+\varepsilon) C +
(1-2\varepsilon)L$. Hence,

\begin{equation*}
-K_{S_d}\sim\sigma^*(-K_{\mathbb{P}^2})-\sum^{k}_{i=1}E_i\equiv 
\left\{%
\aligned 
& (1+\varepsilon)\tilde{C}+(1-2\varepsilon)\tilde{L}-2\varepsilon E_1+(1-\varepsilon) E_2+\varepsilon\sum_{i=3}^{k} E_i \\
&\mbox{if $L$ meets $C$ at $P_2$};\\%
& (1+\varepsilon)\tilde{C}+(1-2\varepsilon)\tilde{L}-2\varepsilon E_1+\varepsilon\sum_{i=2}^{k} E_i \\
&\mbox{otherwise},\\%
\endaligned\right.
\end{equation*}
where $\tilde{C}$ and $\tilde{L}$ are the proper transforms 
of $C$ and $L$, respectively, by $\sigma$. Thus, we have
$$
H\equiv
\left\{%
\aligned 
&\frac{1}{\mu}\left((1+\varepsilon)\tilde{C}+(1-2\varepsilon)\tilde{L}+(a_1-2\varepsilon)
E_1+(a_2+1-\varepsilon)
E_2+\sum_{i=3}^{r} (a_i+\varepsilon)E_i+\varepsilon\sum_{i=r+1}^{k}
E_i\right) \\ &\mbox{if $L$ meets $C$ at $P_2$};\\
&\frac{1}{\mu}\left((1+\varepsilon)\tilde{C}+(1-2\varepsilon)\tilde{L}+(a_1-2\varepsilon)
E_1+\sum_{i=2}^{r} (a_i+\varepsilon)E_i+\varepsilon\sum_{i=r+1}^{k}
E_i\right)
\\ &\mbox{otherwise}.\\
\endaligned\right.
$$
For $0<\varepsilon<\frac{a_1}{2}$,  this defines an $H$-polar cylinder because 
$$S_d\setminus (\tilde{C} \cup \tilde{L}\cup E_1\cup \ldots\cup E_k)\cong \mathbb{P}^2\setminus (C\cup L).$$
\end{proof}

As in  the case where $L$ meets $C$ at $P_2$ in the proof of Lemma~\ref{lemma:deg-3-cylinders-birational-1}, it can happen that we should separately deal with the case where two curves meet at one of the centers of blow ups. However, in such a case, we always obtain a bigger coefficient for the exceptional curve over the center than when it is not the case. For this reason, in the sequel, we always omit the proof for such a special case. The proof for a non-special case  works  almost verbatim for such a special case.

\begin{lemma}
\label{lemma:deg-3-cylinders-birational-2} 
Suppose that
$H$ is of type $B(8-d)$ and
$Z\cong\mathbb{P}^1\times\mathbb{P}^1$. Then $S_d$ contains an
$H$-polar cylinder.
\end{lemma}

\begin{proof} 
Let $E_1,\ldots,E_r$ be the $(-1)$-curves that generate the face $\Delta$. Note that $r=8-d$.
Then
$$
K_{S_d}+\mu H\equiv\sum_{i=1}^{r} a_iE_i
$$
for some positive real numbers $a_1,\ldots,a_r$. The $(-1)$-curves  $E_1,\ldots,E_r$ are disjoint.
Put $P_i=\phi(E_i)$.

Since $r\leqslant 5$, there is an irreducible curve  $C$ of bidegree $(2,1)$  in $\mathbb{P}^1\times\mathbb{P}^1$ passing through the points $P_1,\ldots, P_{r}$. Let $L$ be a curve of bidegree $(0,1)$  in $\mathbb{P}^1\times\mathbb{P}^1$ that is tangent to the curve $C$. Let $P$ be the intersection point of $C$ and $L$.
Then there is a unique curve $M$ of bidegree $(1,0)$  in $\mathbb{P}^1\times\mathbb{P}^1$ passing through the point $P$.

For a positive real number $\varepsilon$ we have
$-K_{\mathbb{P}^1\times\mathbb{P}^1}\equiv(1-\varepsilon) C +
(1+\varepsilon)L+2\varepsilon M$. Hence,
$$
-K_{S_d}\sim\phi^*(-K_{\mathbb{P}^1\times\mathbb{P}^1})-\sum_{i=1}^{r} E_i\equiv (1-\varepsilon)\tilde{C}+(1+\varepsilon)\tilde{L}+2\varepsilon \tilde{M}-\varepsilon\sum_{i=1}^{r} E_i,%
$$
where $\tilde{C}$, $\tilde{L}$, and $\tilde{M}$ are the proper transforms 
of $C$, $L$, and $M$, respectively, by $\phi$. Thus, we have
$$
H\equiv\frac{1}{\mu}\left((1-\varepsilon)\tilde{C}+(1+\varepsilon)\tilde{L}+2\varepsilon \tilde{M}+\sum_{i=1}^{r} (a_i-\varepsilon)E_i\right).
$$
Furthermore, we see immediately that
$$S_d\setminus (\tilde{C} \cup \tilde{L}\cup \tilde{M}\cup E_1\cup \ldots\cup E_r)\cong \mathbb{P}^1\times\mathbb{P}^1\setminus (C\cup L\cup M).$$
By taking $0<\varepsilon<\min\{a_1, \ldots, a_r\}$  we obtain
 an $H$-polar cylinder on $S_d$ (see Example~\ref{example:non-Prokhorov example8-1}).
\end{proof}
\begin{lemma}
\label{lemma:deg-3-cylinders-conic-bundle} Suppose that
 $H$ is of type $C(9-d)$.  Then
$S_d$ contains an $H$-polar cylinder.
\end{lemma}

\begin{proof}
If the contraction $\phi$ is a conic bundle, then, as in \eqref{lemma:curves-in-face},  we may write
$$
K_{S_d}+\mu H\equiv aB+\sum_{i=1}^{m}a_iE_i
$$
where $B$ is an irreducible fiber of $\phi$, $E_i$'s are disjoint $(-1)$-curves in fibers of $\phi$, $a$ is a positive real number, $a_i$'s are non-negative real numbers, and $m=8-d$. We may assume that $a_1\geqslant a_2\geqslant \ldots \geqslant a_m$. 
Let $\phi_1: S_d\to W$ be the birational morphism obtained by contracting the disjoint $(-1)$-curves $E_1,\ldots, E_m$.
Thus $W$  is a smooth del Pezzo surface of degree $8$, hence either $W\cong \mathbb{P}^1\times
 \mathbb{P}^1$   or $W\cong \mathbb{F}_1$.

\bigskip
\textbf{Case 1.}  $a_m\ne 0$ and $W\cong \mathbb{F}_1$.
\medskip

There is a $(-1)$-curve $E$ on $S_d$ whose image by $\phi_1$ is the unique $(-1)$-curve  on $W$.  Let $\psi:W\to\mathbb{P}^2$ be the birational morphism given by contracting $\phi_1(E)$. Put $\sigma=\psi\circ \phi_1$.  Denote the points $\sigma(E_i)$ by $P_i$, $i=1, \ldots, m$, the point $\sigma(E)$ by~$P$, and the line $\sigma (B)$ by $M$.  Note that the line $M$ passes through the point $P$.

Let $C$ be the conic passing through the  points $P_1,\ldots, P_m$. Such a conic exists because $m\leqslant 5$. Let $L$ be a line that passes through the point $P$ and that is tangent to the conic $C$. We may assume that the line $L$ is different from $M$.

 For any real number $\varepsilon$ we have
$-K_{\mathbb{P}^2}\equiv (1-\varepsilon) C +
(1+2\varepsilon+a)L-aM$. Hence,
\[
\begin{split}
-K_{S_d}&\sim\sigma^*(-K_{\mathbb{P}^2})-\sum_{i=1}^{m} E_i-E\\
&\equiv (1-\varepsilon)\tilde{C}+(1+2\varepsilon+a)\tilde{L}+2\varepsilon E-aB -\varepsilon\sum_{i=1}^{m} E_i,\\
\end{split}
\]
where $\tilde{C}$ and $\tilde{L}$ are the proper transforms of $C$ and $L$, respectively, by $\sigma$.
Thus, we have
$$
H\equiv\frac{1}{\mu}\left( (1-\varepsilon)\tilde{C}+(1+2\varepsilon+a)\tilde{L}+2\varepsilon E +\sum_{i=1}^{m} (a_i-\varepsilon)E_i   \right).
$$
By taking a sufficiently small  positive real number $\varepsilon$ we obtain
 an $H$-polar cylinder on $S_d$.

 \bigskip
 \textbf{Case 2.}  Either $a_m\ne 0$ with $W\cong \mathbb{P}^1\times\mathbb{P}^1$ or $a_m= 0$.
 \medskip
 
 We first assume that $W\cong \mathbb{P}^1\times\mathbb{P}^1$.
 Denote the points $\phi_1(E_i)$ by $P_i$, $i=1, \ldots, m-1$, the point $\phi_1(E_m)$ by~$P$, and the curve $\phi_1 (B)$ by $M$.  The curve $M$ is a curve of bidegree $(0,1)$ or $(1,0)$ on $\mathbb{P}^1\times\mathbb{P}^1$.  We may assume that it is of bidegree $(0,1)$.
 
 There is a unique curve  $C$ of bidegree $(1,2)$ passing through the points $P, P_1,\ldots, P_{m-1}$.  There is a curve  $L$
 of bidegree $(1,0)$ that is tangent to $C$. Let $Q$ be the point at which $L$ meets $C$ and let $N$ be the curve of bidegree $(0,1)$ passing through the point $Q$.
 
 For an arbitrary real number $\varepsilon$ we have
$-K_{\mathbb{P}^1\times\mathbb{P}^1}\equiv (1+\varepsilon) C +
(1-\varepsilon)L+(a-2\varepsilon)N-aM$. Hence,
\[
\begin{split}
-K_{S_d}&\sim\phi_1^*(-K_{\mathbb{P}^1\times\mathbb{P}^1})-E_m-\sum_{i=1}^{m-1} E_i\\
&\equiv(1+\varepsilon) \tilde{C} +
(1-\varepsilon)\tilde{L}+(a-2\varepsilon)\tilde{N}-aB+\varepsilon E_m+\sum_{i=1}^{m-1} \varepsilon E_i,\\
\end{split}
\]
where $\tilde{C}$, $\tilde{L}$, and $\tilde{N}$ are the proper transforms of $C$, $L$, $N$, respectively, by $\phi_1$.
Thus, we have
$$
H\equiv\frac{1}{\mu}\left( (1+\varepsilon) \tilde{C} +
(1-\varepsilon)\tilde{L}+(a-2\varepsilon)\tilde{N}+(a_m+\varepsilon) E_m+\sum_{i=1}^{m-1} (a_i+\varepsilon) E_i,  \right).
$$
By taking a sufficiently small  positive real number $\varepsilon$ we obtain
 an $H$-polar cylinder on $S_d$ (see Example~\ref{example:non-Prokhorov example8-1}).
 
 Now we assume that $W\cong \mathbb{F}_1$ and $a_m=0$.
 Let $E_m'$ be the other $(-1)$-curve in the fiber of~$\phi$ containing the $(-1)$-curve $E_m$.  Let $\phi_2:S_d\to\mathbb{P}^1\times\mathbb{P}^1$ be the birational morphism given by contracting the $(-1)$-curves $E_1,\ldots, E_{m-1}, E_m'$. 
 Since $a_m=0$, after replacing $\phi_1$ and $E_m$ by $\phi_2$ and~$E_m'$, we see immediately that the previous argument also works for this case.
  \end{proof}
Theorem~\ref{theorem:main-easy} immediately follows from Theorem~\ref{theorem:smooth-del-Pezzos-degree-1-2-3} and
Lemmas~\ref{lemma:deg-3-cylinders-birational-1}, ~\ref{lemma:deg-3-cylinders-birational-2}, and~\ref{lemma:deg-3-cylinders-conic-bundle}.

\section{Absence of cylinders}
\subsection{The main obstruction}
\label{section:obstruction}
We here refine Remark~\ref{observation}  for a smooth rational surface as below. 

Let $S$ be a smooth rational surface and let $A$ be a big
$\mathbb{R}$-divisor on $S$. Suppose that $S$ contains an $A$-polar
cylinder, i.e., there is an open affine subset $U\subset S$ and an
effective $\mathbb{R}$-divisor $D$ such that $D\equiv
A$, $U=S\setminus\mathrm{Supp}(D)$, and $U\cong \mathbb{A}^1\times Z$ for some smooth rational affine curve $Z$. Put
$D=\sum_{i=1}^{n}a_i C_i$, where each $C_i$ is an irreducible
reduced curve and each $a_i$ is a positive real number. Let $\mu$ be the Fujita invariant of $A$.

As in \eqref{diagram}, the natural projection $p_Z:U\cong \mathbb{A}^1\times Z\to Z$ induces a
rational map $\psi\colon S\dasharrow\mathbb{P}^1$ given by a 
pencil $\mathcal{L}$ on the surface $S$. If the base locus $\mathrm{Bs}(\mathcal{L})$ of the pencil  is non-empty, then it must consist of a single point
because $\psi^{*}(Q)\cong \mathbb{P}^1$ for a general point $Q$ of $\mathbb{P}^1$ and $\mathrm{Supp}(\psi^{*}(Q))\setminus \mathrm{Supp}(\mathrm{Bs}(\mathcal{L}))$ contains an affine line. 

\begin{theorem}
\label{theorem:obstruction} Suppose that  the base locus of $\mathcal{L}$ consists of a
point $P$ on $S$. Then, for every
effective $\mathbb{R}$-divisor $B$ on $S$ such that
$\mathrm{Supp}(B)\subset\mathrm{Supp}(D)$ and $K_{S}+B$ is
pseudo-effective, the log pair $(S, B)$ is  not log canonical at $P$.
In particular, $(S, \mu D)$ is not log canonical at $P$.
\end{theorem}
\begin{proof}
The proof is the same as the explanation for Remark~\ref{observation}.
We may assume that the exceptional divisor of $\pi$ lies over the point $P$.
Since $\mathrm{Supp}(B)\subset\mathrm{Supp}(D)$, the divisor $B$ can be written as
$B=\sum_{i=1}^{n}b_i C_i$,
where $b_i$'s are non-negative real numbers. The remaining parts are exactly  the same as in the explanation for Remark~\ref{observation}.
 \end{proof}

\begin{theorem}
\label{theorem:del-Pezzo-tigers} Let $S_d$ be a smooth del Pezzo
surface of degree $d\leqslant 3$ and let $D$ be an effective
 $\mathbb{R}$-divisor on $S_d$ such that  $D\equiv -K_{S_d}$.
If  the log pair
$(S_d,D)$ is not log canonical at a point~$P$, then there exists a
divisor $T$ in the anticanonical linear system $|-K_{S_d}|$ such
that the log pair $(S_d,T)$ is not log canonical at the point $P$ and  $\mathrm{Supp}(T)\subset\mathrm{Supp}(D)$.
\end{theorem}
\begin{proof}
See \cite[Theorem~1.12]{CheltsovParkWon} for an effective
 $\mathbb{Q}$-divisor. The proof  works verbatim for  an effective
 $\mathbb{R}$-divisor.
\end{proof}

\subsection{Proof of Theorem~\ref{theorem:main-non-existence}}
\label{section:degree-2}

Before we prove 
Theorem~\ref{theorem:main-non-existence}, we introduce two easy results that we use for the proof.

\begin{lemma}
\label{lemma:mult}
Let $S$ be a smooth surface and let $D$ be an effective $\mathbb{R}$-divisor on
$S$.
If the log pair $(S,D)$ is  not  log canonical at  a point $P$, then
$\mathrm{mult}_{P}(D)> 1$.  \end{lemma}
\begin{proof}
For instance, see \cite[Proposition~9.5.13]{La04II}.
\end{proof}

\begin{lemma}
\label{lemma:dp1-2-3} Let $S$ be a smooth del Pezzo surface of
degree $d\leqslant 3$ and let $$D=\sum_{i=1}^na_iD_i$$ be an effective
$\mathbb{R}$-divisor on $S$ such that $D\equiv -K_{S}$, where $D_1,\ldots,D_n$ are irreducible
curves and $a_1,\ldots,a_n$ are positive real numbers. Then
$a_i\leqslant 1$ for each $i=1,\ldots, n$.
\end{lemma}

\begin{proof}
For the case where $d=1$, the statement follows from
\[1=K_S^2=\sum_{i=1}^na_iD_i\cdot(-K_S)\geqslant a_iD_i\cdot(-K_S)\geqslant a_i.\]
For the cases where $d=2$ and $3$, see 
\cite[Lemmas~3.1 and 4.1]{CheltsovParkWon}, respectively.  Their proofs work verbatim for  an effective
 $\mathbb{R}$-divisor.
\end{proof}

Theorem~\ref{theorem:main-non-existence} immediately follows from the following two statements.

\begin{theorem}
\label{theorem:dp1-2-rank-1} Let $S$ be a smooth del Pezzo surface
of degree $d\leqslant 2$. Let $E$ be a $(-1)$-curve on $S$. For a positive real number $a$ the surface $S$ does
not contain any $(-K_S+aE)$-polar cylinder.
\end{theorem}

\begin{proof}
Suppose that there exists an effective $\mathbb{R}$-divisor $D$
such that $D\equiv -K_S+aE$ and  
$S\setminus~\mathrm{Supp}(D)$ is isomorphic to $\mathbb{A}^1\times Z$ for some
affine variety $Z$.

Let $f\colon S\to\overline{S}$ be the contraction of the curve $E$. Put
$\overline{D}=f(D)$. Then $\overline{S}$ is a smooth del Pezzo surface of
degree $d+1\leqslant 3$. Moreover, we have
$\overline{D}\equiv -K_{\overline{S}}$. 

If
$E\subset\mathrm{Supp}(D)$, then
$$
\overline{S}\setminus\mathrm{Supp}(\overline{D})\cong S\setminus
\mathrm{Supp}(D)\cong \mathbb{A}^1\times Z,
$$
 which implies that
$\overline{S}\setminus\mathrm{Supp}(\overline{D})$ is a
$(-K_{\overline{S}})$-polar cylinder on $\overline{S}$. This contradicts
Theorem~\ref{theorem:smooth-del-Pezzos-degree-1-2-3}.  Therefore, 
$E\not\subset\mathrm{Supp}(D)$. 
In particular, $D\cdot E\geqslant 0$ and $a\leqslant 1$.

Put $D=\sum_{i=1}^{n}a_iD_i$, where $D_1,\ldots,D_n$ are
irreducible curves and $a_1,\ldots,a_n$ are positive real
numbers.
None of the  curves $D_1,\ldots,D_n$ are contracted by the
morphism $f$ and
$$
\sum_{i=1}^{n}a_if(D_i)=\overline{D}\equiv -K_{\overline{S}}.
$$
Therefore, we have $a_i\leqslant 1$ for each $i=1,\ldots, n$ by  Lemma~\ref{lemma:dp1-2-3}. 
Since $a\leqslant 1$ too, by the second case in  Remark~\ref{observation} the linear system $\mathcal {L}$ associated with the cylinder $S\setminus\mathrm{Supp}(D)$ has a base point, say $P$. Due to 
Theorem~\ref{theorem:obstruction} for
every effective $\mathbb{R}$-divisor $B$ on  $S$ such
that $K_{S}+B$ is pseudo-effective and
$\mathrm{Supp}(B)\subset\mathrm{Supp}(D)$, the log pair $(S, B)$ is not log canonical at $P$. In particular,  $(S,D)$ is not log canonical at the point $P$.

The inequality 
$$
1>1-a=\left(-K_{S}+aE\right)\cdot E=D\cdot E\geqslant\mathrm{mult}_{P}(D)\mathrm{mult}_{P}(E)%
$$
 and Lemma~\ref{lemma:mult} show that $P$ lies outside  $E$. Therefore, $(\overline{S},\overline{D})$ is not log canonical at $f(P)$.

Let $\overline{T}$ be the unique divisor in $|-K_{\overline{S}}|$ that is
singular at $f(P)$. Denote by $T$ its proper transform on the
surface $S$. Since $\overline{D}\equiv -K_{\overline{S}}$ and
$(\overline{S},\overline{D})$ is not log canonical at the point $f(P)$, it
follows from Theorem~\ref{theorem:del-Pezzo-tigers} that
$(\overline{S},\overline{T})$ is not log canonical at $f(P)$ and  $\mathrm{Supp}(\overline{T})\subset\mathrm{Supp}(\overline{D})$. 
Hence,  $\mathrm{Supp}(T)\subset\mathrm{Supp}(D)$.

For every non-negative real number $\mu$, put
$D_{\mu}=(1+\mu)D-\mu T$ and $\overline{D}_{\mu}=(1+\mu)\overline{D}-\mu
\overline{T}$. Since $-K_{\overline{S}}\cdot\overline{T}=K_{\overline{S}}^2\leqslant 3$, the divisor $T$ consists of at most
$3$ irreducible
components.  Therefore, $D\ne T$ because the divisor $D$ has at least 
$8$ components 
by \eqref{equation:KPZ-r}.  Put
$$
\nu=\mathrm{sup}\Big\{\mu\in\mathbb{R}_{\geqslant 0}\ \Big\vert\ D_{\mu}\ \text{is effective}\Big\}.%
$$
Then  $\mathrm{Supp}(T)\not\subset\mathrm{Supp}(D_\nu)$ and $\mathrm{Supp}(\overline{T})\not\subset\mathrm{Supp}(\overline{D}_\nu)$. In particular, we have
$\nu>0$ since $\mathrm{Supp}(T)\subset\mathrm{Supp}(D)$. 

We have $\overline{D}_{\mu}\equiv\overline{D}\equiv
\overline{T}\equiv -K_{\overline{S}}$ for each  real number $\mu$. This implies that
$$
D_{\mu}\equiv -K_{S}+a_\mu E
$$
for some real number $a_\mu$. Note that $a_\mu$ is either linear
or constant in $\mu$. 

Suppose that $a_\nu\geqslant 0$. Then $K_{S}+D_\nu$ is
pseudo-effective. Therefore, the log pair $(S,D_\nu)$ is not log
canonical at the point $P$ by Theorem~\ref{theorem:obstruction}.
Then $(\overline{S},\overline{D}_\nu)$ is not log canonical at $f(P)$. The
latter contradicts Theorem~\ref{theorem:del-Pezzo-tigers} because
 $\mathrm{Supp}(\overline{T})\not\subset\mathrm{Supp}(\overline{D}_\nu)$ by the choice of $\nu$.

Suppose that $a_\nu<0$. Since 
$a_0=a>0$, there exists a positive real number
$\lambda\in(0,\nu)$ such that $a_\lambda=0$. It follows from $\lambda<\nu$ that
$\mathrm{Supp}(T)\subset\mathrm{Supp}(D_\lambda)$ and
$\mathrm{Supp}(D_\lambda)=\mathrm{Supp}(D)$. 
Therefore,
$$
S\setminus\mathrm{Supp}(D_\lambda)\cong S\setminus
\mathrm{Supp}(D)\cong \mathbb{A}^1\times Z
$$
is a  cylinder.
However, this contradicts Theorem~\ref{theorem:smooth-del-Pezzos-degree-1-2-3} because  $a_\lambda=0$, i.e.,
 $D_\lambda\equiv-K_{S}$.
\end{proof}

\begin{theorem}
\label{theorem:dp1-rank-2} Let $S$ be a smooth del Pezzo surface
of degree $1$. Let $E$ and $F$ be two disjoint  $(-1)$-curves on $S$. The surface
$S$ contains no $(-K_{S}+aE+bF)$-polar cylinder for any positive real numbers $a$ and $b$.
\end{theorem}

\begin{proof}
Suppose that there exists an effective $\mathbb{R}$-divisor $D$
such that $D\equiv -K_{S}+aE+bF$ and such that $S\setminus
\mathrm{Supp}(D)$ is isomorphic to $\mathbb{A}^1\times Z$ for some
affine variety $Z$. In the following we  seek for a contradiction.

Let $g\colon S\to\hat{S}$ be the contraction of the curve $E$. Put
$\hat{D}=g(D)$ and $\hat{F}=g(F)$. Then $\hat{S}$
is a smooth del Pezzo surface of degree $2$, 
 $\hat{F}$ is a $(-1)$-curve, and
$\hat{D}\equiv -K_{\hat{S}}+b\hat{F}$. This implies that
$E\not\subset\mathrm{Supp}(D)$. Indeed, if
$E\subset\mathrm{Supp}(D)$, then
$$
\hat{S}\setminus\mathrm{Supp}(\hat{D})\cong S\setminus
\mathrm{Supp}(D)\cong \mathbb{A}^1\times Z
$$
is a $\hat{D}$-polar cylinder on $\hat{S}$. This is impossible
by Theorem~\ref{theorem:dp1-2-rank-1}.
Similarly, we see that
$F\not\subset\mathrm{Supp}(D)$. Therefore, $D\cdot E\geqslant 0$, $D\cdot F\geqslant 0$ and $a\leqslant 1$, $b\leqslant 1$.

Write $D=\sum_{i=1}^{n}a_iD_i$, where $D_1,\ldots,D_n$ are
irreducible curves and $a_1,\ldots,a_n$ are positive real
numbers. 

Let $f\colon S\to\overline{S}$ be the contraction of the curves $E$ and
$F$. Put $\overline{D}=f(D)$. Then $\overline{S}$ is a smooth cubic surface
and $\overline{D}\equiv -K_{\overline{S}}$. None of the  curves $D_1,\ldots,D_n$ are contracted by the
morphism~$f$ and
$$
\sum_{i=1}^{n}a_if(D_i)=\overline{D}\equiv -K_{\overline{S}}.
$$
Therefore, we have $a_i\leqslant 1$ for each $i=1,\ldots, n$ by  Lemma~\ref{lemma:dp1-2-3}.  Because $a, b\leqslant 1$, the second case in Remark~\ref{observation} implies that the linear system $\mathcal {L}$ associated with the cylinder $S\setminus\mathrm{Supp}(D)$ has a base point, say $P$.
By
Theorem~\ref{theorem:obstruction} for
every effective $\mathbb{R}$-divisor $B$ on $S$ such
that $K_{S}+B$ is pseudo-effective and
$\mathrm{Supp}(B)\subset\mathrm{Supp}(D)$, the log pair $(S, B)$ is not log canonical at $P$ . In particular,  $(S,D)$ is not log canonical at the point $P$.

We claim that $P$ belongs to neither $E$ nor $F$. Indeed, if $P\in E$, then
$$
1> 1-a=\left(-K_{S}+aE\right)\cdot E=D\cdot E\geqslant\mathrm{mult}_{P}(D)>1%
$$
by Lemma~\ref{lemma:mult}. This shows that $P\not\in E$.
Similarly, we see that $P\not\in F$. 
Therefore, the birational morphism $f$ is an isomorphism in a
neighborhood of the point $P$. In particular, the log pair
$(\overline{S},\overline{D})$ is not log canonical at $f(P)$.

Let $\overline{T}$ be the unique divisor in $|-K_{\overline{S}}|$ that is
singular at $f(P)$. Denote by $T$ its proper transform on the
surface $S$. Since $\overline{D}\equiv -K_{\overline{S}}$ and
$(\overline{S},\overline{D})$ is not log canonical at the point $f(P)$, it
follows from Theorem~\ref{theorem:del-Pezzo-tigers} that
$(\overline{S},\overline{T})$ is not log canonical at $f(P)$ and  $\mathrm{Supp}(\overline{T})\subset\mathrm{Supp}(\overline{D})$. 
Hence,  $\mathrm{Supp}(T)\subset\mathrm{Supp}(D)$.

For every non-negative real number $\mu$, put
$D_{\mu}=(1+\mu)D-\mu T$ and $\overline{D}_{\mu}=(1+\mu)\overline{D}-\mu
\overline{T}$. Since $-K_{\overline{S}}\cdot\overline{T}=K_{\overline{S}}^2=3$, the divisor $T$ consists of at most
$3$ irreducible
components.  Therefore, $D\ne T$ because the divisor $D$ has at least 
$9$ components 
by \eqref{equation:KPZ-r}.  Put
$$
\nu=\mathrm{sup}\Big\{\mu\in\mathbb{R}_{\geqslant 0}\ \Big\vert\ D_{\mu}\ \text{is effective}\Big\}.%
$$
Then  $\mathrm{Supp}(T)\not\subset\mathrm{Supp}(D_\nu)$ and $\mathrm{Supp}(\overline{T})\not\subset\mathrm{Supp}(\overline{D}_\nu)$. In particular, we have
$\nu>0$ since $\mathrm{Supp}(T)\subset\mathrm{Supp}(D)$. 

We have $\overline{D}_{\mu}\equiv\overline{D}\equiv
\overline{T}\equiv -K_{\overline{S}}$ for each  real number $\mu$. This implies that
$$
D_{\mu}\equiv -K_{S}+a_\mu E+b_\mu F
$$
for some real numbers $a_\mu$ and $b_\mu$. From 
$-K_S+E+F\equiv f^*(\overline{D}_\mu)=(1+\mu)f^*(\overline{D})-\mu f^*(\overline{T})$  and $a_0=a$, $b_0=b$ we obtain
$$
\left\{%
\aligned
&a_\mu=\left(\mult_{f(E)}(\overline{T})-\mult_{f(E)}(\overline{D})\right)\mu+a\\
&b_\mu=\left(\mult_{f(F)}(\overline{T})-\mult_{f(F)}(\overline{D})\right)\mu+b.\\
\endaligned\right.%
$$

Suppose that $a_\nu\geqslant 0$ and $b_\nu\geqslant 0$. Then
$K_{S}+D_\nu$ is pseudo-effective, and hence the log pair $(S,D_\nu)$ is not
log canonical at the point $P$ by Theorem~\ref{theorem:obstruction}.
Then $(\overline{S},\overline{D}_\nu)$ is not log canonical at $f(P)$.  Since 
$\mathrm{Supp}(\overline{T})\not\subset\mathrm{Supp}(\overline{D}_\nu)$, this contradicts Theorem~\ref{theorem:del-Pezzo-tigers}. 

Suppose that either $a_\nu< 0$ or $b_\nu< 0$. Since $a_0=a>0$ and $b_0=b>0$, there is a real number $\lambda \in (0, \nu)$
such that either $a_\lambda=0$, $b_\lambda\geqslant 0$ or $a_\lambda\geqslant 0$, $b_\lambda=0$. Without loss of generality we may assume that $a_\lambda=0$. Since $\lambda <\nu$,  $\mathrm{Supp}(T)\subset\mathrm{Supp}(D_\lambda)=\mathrm{Supp}(D)$. Therefore, $S\setminus \mathrm{Supp}(D_\lambda)=S\setminus \mathrm{Supp}(D)$ is a cylinder. However,
this contradicts either Theorem~\ref{theorem:smooth-del-Pezzos-degree-1-2-3} or Theorem~\ref{theorem:dp1-2-rank-1} since
$$
D_{\lambda}\equiv -K_{S}+b_\lambda F.
$$
\end{proof}

\section{Cylinders in del Pezzo surfaces of small degrees}

\subsection{Del Pezzo surface of degree 2 without a cuspidal anticanonical divisor}
A smooth quartic plane curve can have at most twenty four inflection points. There may be two kinds of inflection points on a smooth quartic curve. One is a point at which its tangent line intersects the quartic with multiplicity $3$, and the other with multiplicity $4$. The former is called an ordinary inflection point and the latter a hyperinflection point. 
These inflection points can be spotted with the Hessian curve of the given quartic curve. The Hessian curve intersects the quartic curve transversally at ordinary inflection points and meets the quartic curve at hyperinflection points with multiplicity $2$. Since the degree of the Hessian curve is $6$, we have
\[\mbox{the number of ordinary inflection points } + 2\times \mbox{ the number of hyperinflection points } =24.\]
Therefore, a smooth quartic plane curve has exactly twelve hyperinflection points if it contains no  ordinary inflection point.

A  smooth del Pezzo surface of degree $2$ is a double cover of $\mathbb{P}^2$ ramified along a smooth plane quartic curve.
An effective anticanonical divisor on a smooth del Pezzo surface of degree $2$ is given by the pull-back of  a line on $\mathbb{P}^2$ via the double covering map. 
An effective anticanonical divisor that is a cuspidal rational curve is given exactly by  the pull-back of the tangent line at an ordinary inflection point. The pull-back of the tangent line at a hyperinflection point is an effective anticanonical divisor that is a tacnodal curve, i.e., two $(-1)$-curves intersecting at a single point tangentially.
Consequently, a smooth del Pezzo surface of degree $2$ contains twelve effective anticanonical divisors that are tacnodal curves if its anticanonical linear system contains no cuspidal rational curve. Each of the twelve tacnodal curves consists of two distinct  $(-1)$-curves intersecting at a single point tangentially.  These twenty four $(-1)$-curves are distinct.

In fact, there are exactly two quartic plane curves without any ordinary inflection point ( \cite{E45}, \cite{KuKo79}). One is the Fermat quartic, i.e., the curve defined by
\[x^4+y^4+z^4=0,\]
 and the other is the curve defined by
\[
x^4+y^4+z^4+3(x^2y^2+y^2z^2+z^2x^2)=0.\]
As explained above, these have exactly twelve  hyperinflection points.
The del Pezzo surfaces of degree~$2$ corresponding these two quartic curves are the only del Pezzo surfaces of degree $2$ whose anticanonical linear systems contain no cuspidal rational curves. 

\subsection{Cylinders in del Pezzo surfaces of degree 2}
\label{section:cylinders-degree-2}

In order to prove Theorem~\ref{theorem:main-hard}, let
$S$ be a smooth del Pezzo surface of degree $2$ and let $H$ be an ample
$\mathbb{R}$-divisor on~$S$. Let $\mu$ and $r$ be the Fujita invariant and the Fujita rank of $H$.
Denote by $\Delta$ the Fujita face of $H$. Let $\phi: S\to Z$ be the contraction given by  $\Delta$.

We first consider ample $\mathbb{R}$-divisors of type $B(r)$. Let $E_1,\ldots, E_r$ be the $r$ disjoint $(-1)$-curves that generate the face $\Delta$.
We may then write 
\begin{equation}\label{equation:coefficient}
K_{S}+\mu H\equiv \sum^{r}_{i=1}a_iE_i
\end{equation}
for some positive real numbers $a_1, \ldots, a_r$ (see
\eqref{remark:curves-in-face-birational}). 
\begin{theorem}\label{theorem:BR3-7}
If the ample $\mathbb{R}$-divisor $H$ is of type $B(r)$ with $3\leqslant r\leqslant 7$, then $S$ contains an $H$-polar cylinder.
\end{theorem}
\begin{proof}
The proof is divided into two cases. One is the case when $S$ has a cuspidal rational curve in $|-K_S|$, and the other is the case when it does not.

\bigskip

\textbf{Case 1.} The surface $S$ has no cuspidal rational curve in $|-K_S|$. 
\medskip

In this case, as mentioned in the previous subsection, $S$  has exactly  twelve pairs of $(-1)$-curves $\{C_i, C'_i\}$, $i=1,\ldots, 12$ such that each $C_i+C'_i$ is a tacnodal anticanonical divisor.

Choose one tacnodal anticanonical divisor, say $C_1+C_1'$. Since we have more than seven tacnodal anticanonical divisors, we may assume that 
\begin{itemize}
\item none of $E_i$ are $C_1$ or $C_1'$;
\item if $r=6$, then neither $\phi(C_1)$ nor $\phi(C_1')$ is a $(-1)$-curve.
\end{itemize}  Let $m$ be the number of the curves $E_i$'s intersecting $C_1$ and $m'$ be the number of the curves $E_i$'s intersecting $C_1'$. We may assume that $E_1, \ldots, E_m$ intersect $C_1$ and that $E_{m+1},\ldots, E_r$ meet $C_1'$. Furthermore, we may assume that $m\geqslant m'$.  Note that $m+m'=r$ . Furthermore,  by the assumption above, $2\leqslant m\leqslant 5$.
Let $\overline{C}_1$ and $\overline{C}_1'$ be the images $C_1$ and $C_1'$ by $\phi$, respectively.
The curve $\overline{C}_1$ is an $(m-1)$-curve on the del Pezzo surface $Z$. Since $m\geqslant 2$, the complete linear system 
$|\overline{C}_1|$ induces a birational morphism $\psi$ of $Z$ into $\mathbb{P}^{m}$. Furthermore, its image is isomorphic to 
\[\left\{%
\aligned 
& \mathbb{P}^2 \mbox{ for $m=2$};\\
& \mathbb{P}^1\times \mathbb{P}^1\cong \mbox{a smooth quadratic  surface} \subset \mathbb{P}^3 \mbox{ for $m=3$};\\
& \mathbb{F}_1\cong\mbox{ a smooth rational normal scroll of degree $3$ in $\mathbb{P}^4$ for $m=4$};\\
&  \mathbb{P}^2\cong\mbox{ a Veronese surface  in $\mathbb{P}^5$ for $m=5$}.\\
\endaligned\right.\]
Put $\sigma=\psi\circ\phi$.

For $m=2$ and $5$, we have $(7-r)$ disjoint $(-1)$-curves on $Z$ that are contracted by $\psi$. These curves do not intersect 
$\overline{C}_1$ but they meet $\overline{C}_1'$ since $\overline{C}_1+\overline{C}_1'$ is an anticanonical divisor of $Z$. 
Let $F_{1},\ldots, F_{7-r}$ be
the pull-backs of these $(7-r)$ disjoint $(-1)$-curves by $\phi$. Then the curve $C_1$ intersects exactly $m$ curves $E_1,\ldots E_m$ and the curve $C_1'$ meets 
the other $m'$ curves $E_{m+1},\ldots E_r$ and all the $(7-r)$ curves $F_{1},\ldots, F_{7-r}$.

For $m=2$, $\sigma(C_1)$ is a line and $\sigma(C_1')$ is a conic  in $\mathbb{P}^2$. For $m=5$, $\sigma(C_1)$ is a conic and $\sigma(C_1')$ is a line  in $\mathbb{P}^2$. They intersect tangentially at a single point.
Therefore, 
\[-K_{S}\equiv (1-a\varepsilon)C_1+(1+b\varepsilon)C_1'-a\varepsilon\sum_{i=1}^{m}E_i+b\varepsilon\sum_{i=m+1}^{r}E_i+b\varepsilon\sum_{i=1}^{7-r}F_i,\]
and hence
\[\mu H\equiv (1-a\varepsilon)C_1+(1+b\varepsilon)C_1'+\sum_{i=1}^{m}(a_i-a\varepsilon)E_i+\sum_{i=m+1}^{r}(a_i+b\varepsilon)E_i+b\varepsilon\sum_{i=1}^{7-r}F_i,\]
where $a=2$, $b=1$ if $m=2$ and $a=1$, $b=2$ if $m=5$.
For a sufficiently small positive real number $\varepsilon$, $H$ yields a cylinder because 
$$S\setminus (C_1 \cup C_1'\cup E_1\cup \ldots \cup E_r\cup F_1\cup\ldots\cup F_{7-r})\cong \mathbb{P}^2\setminus (\sigma(C_1)\cup\sigma(C_1')).$$ 

For $m=3$ and $4$, we have $(6-r)$ disjoint $(-1)$-curves on $Z$ that are contracted by $\psi$. These curves do not intersect 
$\overline{C}_1$ but they meet $\overline{C}_1'$ since $\overline{C}_1+\overline{C}_1'$ is an anticanonical divisor of $Z$. 
Again we let $F_{1},\ldots, F_{6-r}$ be
the pull-backs of these $(6-r)$ disjoint $(-1)$-curves by $\phi$. Then the curve $C_1$ intersects exactly $m$ curves $E_1,\ldots E_m$ and the curve $C_1'$ meets 
the other $m'$ curves $E_{m+1},\ldots E_r$ and all the $(6-r)$ curves $F_{1},\ldots, F_{6-r}$.

For $m=3$, $\sigma(C_1)$ is an irreducible curve of bidegree $(1,1)$ in $\mathbb{P}^1\times \mathbb{P}^1$ since its self-intersection number is $2$. The irreducible curve $\sigma(C_1')$ is also of bidegree $(1,1)$ because $\sigma(C_1)+\sigma(C_1')$ is an anticanonical divisor of $\mathbb{P}^1\times \mathbb{P}^1$.  They intersect tangentially at a single point $Q$. Let $L_1$ and $L_2$ be the curves of bidegrees  $(1,0)$ and $(0,1)$, respectively, passing through $Q$. Then,
\[-K_{S}\equiv (1-2\varepsilon)C_1+(1+\varepsilon)C_1'+\varepsilon(\tilde{L}_1+\tilde{L}_2)-2\varepsilon\sum_{i=1}^{3}E_i+\varepsilon\sum_{i=4}^{r}E_i+\varepsilon\sum_{i=1}^{6-r}F_i,\]
where $\tilde{L}_1$ and $\tilde{L}_2$ are the proper transforms of $L_1$ and $L_2$ by $\sigma$, respectively.
Therefore 
\[\mu H\equiv (1-2\varepsilon)C_1+(1+\varepsilon)C_1'+\varepsilon(\tilde{L}_1+\tilde{L}_2)+\sum_{i=1}^{3}(a_i-2\varepsilon)E_i+\sum_{i=4}^{r}(a_i+\varepsilon)E_i+\varepsilon\sum_{i=1}^{6-r}F_i.\]
For a sufficiently small positive real number $\varepsilon$, we obtain an $H$-polar cylinder because 
$$S\setminus (C_1 \cup C_1'\cup \tilde{L}_1\cup \tilde{L}_2\cup E_1\cup \ldots \cup E_r\cup F_1\cup\ldots\cup F_{6-r})\cong \mathbb{P}^1\times \mathbb{P}^1\setminus (\sigma(C_1)\cup\sigma(C_1')\cup L_1\cup L_2)$$ 
(see Example~\ref{example:non-Prokhorov example8-0}).

For $m=4$, $\sigma(C_1)$ is a 3-curve in $\mathbb{F}_1$. The irreducible curve $\sigma(C_1')$ is a 1-curve intersecting $\sigma(C_1)$ at a single point $Q$ tangentially. Let $M$ be the $0$-curve passing through the point $Q$.
Example~\ref{example:non-Prokhorov example8-3} shows that 
\[-K_{S}\equiv (1-\varepsilon)C_1+(1+\varepsilon)C_1'+\varepsilon\tilde{M}-\varepsilon\sum_{i=1}^{4}E_i+\varepsilon\sum_{i=5}^{r}E_i+\varepsilon\sum_{i=1}^{6-r}F_i\]
and
\[\mu H\equiv (1-\varepsilon)C_1+(1+\varepsilon)C_1'+\varepsilon\tilde{M}+\sum_{i=1}^{4}(a_i-\varepsilon)E_i+\sum_{i=5}^{r}(a_i+\varepsilon)E_i+\varepsilon\sum_{i=1}^{6-r}F_i,\]
where $\tilde{M}$ is the proper transform of $M$. We see also from Example~\ref{example:non-Prokhorov example8-3}  that $H$ defines a cylinder with a sufficiently small positive real number $\varepsilon$.

\bigskip
\textbf{Case 2.}  The surface $S$ possesses a cuspidal rational curve $C$ in $|-K_S|$.
\medskip

Let $P$ be the point at which the curve $C$ has the cusp. Each $E_i$ intersects the curve $C$ at a single smooth point. This cuspidal curve $C$ plays a key role in constructing $H$-polar cylinders case by case, according to $r$.

\bigskip
\textbf{Subcase 1.}  $r=3$.
\medskip

In this subcase, the surface $S$ has five  $0$-curves $F_1,\ldots, F_5$ such that
\begin{itemize}
\item they pass through $P$;
\item they do not meet each other outside $P$;
\item they are disjoint from   the curves $E_1$, $E_2$ and $E_3$.
\end{itemize}
(For the better understanding of this construction, see Example~\ref{example:non-Prokhorov example5-4} after contracting the $(-1)$-curves
$E_1$, $E_2$, $E_3$ to a smooth del Pezzo surface of degree $5$.)

Let $\pi:\tilde{S}\to S$ be the blow up at the point $P$ and let $E$ be the exceptional curve of $\pi$. Then $\tilde{S}$ is a weak del Pezzo surface of degree $1$ and it has exactly one $(-2)$-curve, the proper transform $\tilde{C}$ of $C$. Denote
 the proper transforms on $\tilde{S}$ of the curves
$E_1$, $E_2$, $E_3$, $F_1, \ldots, F_5$ by $\tilde{E}_1$, $\tilde{E}_2$, $\tilde{E}_3$, $\tilde{F}_1, \ldots, \tilde{F}_5$. Since these $(-1)$-curves are disjoint, they give us 
 a contraction $\psi:\tilde{S}\to \mathbb{P}^2$.
 
 Since $\psi(E)$ is a conic and $\psi(\tilde{C})$ is a line on $\mathbb{P}^2$, we immediately see that 
\[-K_{\tilde{S}}\equiv (1-2\varepsilon)\tilde{C}+(1+\varepsilon)E+\varepsilon\sum_{i=1}^5\tilde{F}_i-2\varepsilon(\tilde{E}_1+\tilde{E}_2+\tilde{E}_3), \]
and hence
\[-K_{S}\equiv (1-2\varepsilon)C+\varepsilon\sum_{i=1}^5F_i-2\varepsilon(E_1+E_2+E_3).\]
Therefore, 
\[\mu H\equiv (1-2\varepsilon)C+\varepsilon\sum_{i=1}^5F_i+(a_1-2\varepsilon)E_1+(a_2-2\varepsilon)E_2+(a_3-2\varepsilon)E_3.\]
For a sufficiently small positive real number $\varepsilon$, this is an $H$-polar cylinder because 
$$S\setminus (C \cup F_1\cup F_2\cup F_3\cup F_4\cup F_5\cup E_1\cup E_2\cup E_3)\cong \mathbb{P}^2\setminus (\psi(E)\cup\psi(\tilde{C})).$$ Note that  the conic $\psi(E)$ and the line $\psi(\tilde{C})$ meet tangentially.

\bigskip
\textbf{Subcase 2.}  $r=4$.
\medskip

In this subcase, the surface $S$ has three  $0$-curves $F_1, F_2, F_3$ such that \begin{itemize}
\item they pass through $P$;
\item they do not intersect  each other outside $P$;
\item they are disjoint from   the curves $E_1$, $E_2$, $E_3$, $E_4$.
\end{itemize}
 In addition, it has two $1$-curves $G_1, G_2$
 such that \begin{itemize}
\item they intersect $C$ only at $P$;
\item they do not meet   each other outside $P$;
\item they are disjoint from   the curves $E_1$, $E_2$, $E_3$, $E_4$.
\end{itemize}
(See Example~\ref{example:non-Prokhorov example6-1}  after contracting the $(-1)$-curves
$E_1$, $E_2$, $E_3$, $E_4$ to a smooth del Pezzo surface of degree $6$.)

Let $\pi:\tilde{S}\to S$ be the blow up at the point $P$ and let $E$ be the exceptional curve of $\pi$. Then $\tilde{S}$ is a weak del Pezzo surface of degree $1$ and it has exactly one $(-2)$-curve, the proper transform $\tilde{C}$ of $C$. Denote
 the proper transforms on $\tilde{S}$ of the curves
$E_1,\ldots, E_4$, $F_1$, $F_2$, $F_3$, $G_1$, $G_2$ by $\tilde{E}_1,\ldots, \tilde{E}_4$, $\tilde{F}_1$, $\tilde{F}_2$, 
$\tilde{F}_3$, $\tilde{G}_1$, $\tilde{G}_2$. 
Contracting the seven $(-1)$-curves $\tilde{E}_1,\ldots, \tilde{E}_4$, $\tilde{F}_1$, $\tilde{F}_2$, 
$\tilde{F}_3$, we obtain a birational morphism $\psi:\tilde{S}\to \mathbb{P}^1\times\mathbb{P}^1$.
 Since $\psi(E)$ and $\psi(\tilde{C})$ are curves of bidegree $(1,1)$ on $\mathbb{P}^1\times\mathbb{P}^1$ and $\psi(\tilde{G}_1)$,
 $\psi(\tilde{G}_2)$ are curves of bidegrees $(1,0)$ and $(0,1)$, respectively, on $\mathbb{P}^1\times\mathbb{P}^1$, 
 \[-K_{\tilde{S}}\equiv (1-2\varepsilon)\tilde{C}+(1+\varepsilon)E+
\varepsilon(\tilde{F}_1+\tilde{F}_2+\tilde{F}_3)+
\varepsilon(\tilde{G}_1+\tilde{G}_2)
-2\varepsilon\sum_{i=1}^4\tilde{E}_i, \]
and hence
\[-K_{S}\equiv (1-2\varepsilon)C+
\varepsilon(F_1+F_2+F_3)+
\varepsilon(G_1+G_2)
-2\varepsilon\sum_{i=1}^4E_i. \]
Therefore, 
\[\mu H\equiv (1-2\varepsilon)C+
\varepsilon(F_1+F_2+F_3)+
\varepsilon(G_1+G_2)+
\sum_{i=1}^4(a_i-2\varepsilon)E_i.\]
For a sufficiently small positive real number $\varepsilon$, this defines an $H$-polar cylinder because 
$$S\setminus (C \cup F_1\cup F_2\cup F_3\cup G_1\cup G_2\cup E_1\cup E_2\cup E_3\cup E_4)\cong \mathbb{P}^1\times\mathbb{P}^1\setminus (\psi(E)\cup\psi(\tilde{C})\cup \psi(\tilde{G}_1)\cup \psi(\tilde{G}_2)).$$
Note that  the curves $\psi(E)$ and $\psi(\tilde{C})$ meet tangentially at one point and the curves $\psi(\tilde{G}_1)$, 
$\psi(\tilde{G}_2)$ pass through this point (see Example~\ref{example:non-Prokhorov example8-0}).

\bigskip
\textbf{Subcase 3.}  $r=5$.
\medskip

There are two $0$-curves $L_1$, $L_2$ such that \begin{itemize}
\item they pass through $P$;
\item they do not intersect  each other outside $P$;
\item they are disjoint from   the curves $E_1,\ldots, E_5$.
\end{itemize}
 In addition, there is a unique $1$-curve 
$T$ that meets $C$ only at the point $P$ and that does not intersect any of $E_i$'s (see Example~\ref{example:non-Prokhorov example7-1}).

By contracting  the $(-1)$-curves $E_1,\ldots, E_5$,  we immediately see from Example~\ref{example:non-Prokhorov example7-1} that
\[-K_{S}\equiv(1-\varepsilon)C+\varepsilon L_1 +\varepsilon L_2 +\varepsilon T-\varepsilon\sum^5_{i=1}E_i, \]
and hence 
\[\mu H\equiv
(1-\varepsilon)C+\varepsilon L_1 +\varepsilon L_2 +\varepsilon T+\sum^5_{i=1}(a_i-\varepsilon)E_i.\]
Example~\ref{example:non-Prokhorov example7-1}  shows that for a sufficiently small positive real number  $\varepsilon$, this defines an $H$-polar cylinder.

\bigskip
\textbf{Subcase 4.}  $r=6$ and $Z\cong \mathbb{P}^1\times\mathbb{P}^1$.
\medskip

There are  exactly two $0$-curves $F_1, F_2$ passing through the point $P$ and not intersecting any of $E_i$'s.
(By contracting $E_1,\ldots, E_6$  into the surface $\mathbb{P}^1\times \mathbb{P}^1$, we can easily detect such $0$-curves.) 

Let $\pi:\tilde{S}\to S$ be the blow up at the point $P$ and let $E$ be the exceptional curve of $\pi$. Then $\tilde{S}$ is a weak del Pezzo surface of degree $1$ and it has exactly one $(-2)$-curve, the proper transform $\tilde{C}$ of $C$. Denote
 the proper transforms on $\tilde{S}$ of the curves
$E_1,\ldots, E_6$, $F_1$, $F_2$ by $\tilde{E}_1,\ldots, \tilde{E}_6$, $\tilde{F}_1$, $\tilde{F}_2$. 
Contracting the $(-1)$-curves $\tilde{E}_1,\ldots, \tilde{E}_6$, $\tilde{F}_1$, $\tilde{F}_2$, 
we obtain a birational morphism $\psi:\tilde{S}\to \mathbb{P}^2$.
Note that $\psi(C)$ and  $\psi (E)$ are a conic and a line meeting tangentially on $\mathbb{P}^2$. 
Therefore, 
\[-K_{S}\equiv (1-\varepsilon)C+2\varepsilon(F_1+F_2)-\varepsilon\sum^6_{i=1}E_i, \]
and hence 
\[\mu H\equiv
(1-\varepsilon)C+2\varepsilon(F_1+F_2)+\sum^6_{i=1}(a_i-\varepsilon)E_i.\]
For a sufficiently small positive real number $\varepsilon$, this defines an $H$-polar cylinder since
$$S\setminus (C \cup F_1\cup F_2\cup E_1\cup\ldots\cup E_6)\cong \mathbb{P}^2\setminus (\psi(\tilde{C})\cup \psi(E)).$$

\bigskip
\textbf{Subcase 5.}  $r=6$ and $Z\cong \mathbb{F}_1$.
\medskip

There is a unique $0$-curve $L$ passing through the point $P$ and not meeting any of $E_i$'s. In addition, there is a unique $1$-curve 
$T$ that intersects $C$ only at the point $P$ and that does not intersect any of $E_i$'s.
(By contracting $E_1,\ldots, E_6$  into the Hirzebruch surface $\mathbb{F}_1$, we can easily detect such curves.)

Example~\ref{example:non-Prokhorov example8-2} shows that
\[-K_{S}\equiv (1-\varepsilon)C+\varepsilon L +2\varepsilon T-\varepsilon\sum^6_{i=1}E_i, \]
and hence 
\[\mu H\equiv
(1-\varepsilon)C+\varepsilon L +2\varepsilon T+\sum^6_{i=1}(a_i-\varepsilon)E_i.\]
It also shows that
for a sufficiently small positive real number  $\varepsilon$, this defines an $H$-polar cylinder since
$$S\setminus (C \cup L\cup T\cup E_1\cup\ldots\cup E_6)\cong \mathbb{F}_1\setminus (\phi(C)\cup \phi(L)\cup \phi(T)).$$ 

\bigskip
\textbf{Subcase 6.}  $r=7$.
\medskip

By contracting $E_1,\ldots, E_7 $ we obtain a birational morphism $\pi$ of $S$ onto the projective plane $\mathbb{P}^2$. 
The curve $\pi(C)$ is a cuspidal cubic curve passing through all the points $\pi(E_i)$'s.
Let $T$ be the Zariski tangent line to the curve $\pi(C)$ at its cuspidal point.
We immediately see that
\[-K_{S}\equiv (1-\varepsilon)\pi^*(\pi(C))+3\varepsilon\pi^*(T)-\sum^7_{i=1}E_i\equiv
(1-\varepsilon)C+3\varepsilon\tilde{T}-\varepsilon\sum^7_{i=1}E_i, \]
where $\tilde{T}$ is the proper transform of $T$ by $\pi$,
and hence 
\[\mu H \equiv
(1-\varepsilon)C+3\varepsilon\tilde{T}+\sum^7_{i=1}(a_i-\varepsilon)E_i.\]
For a sufficiently small positive real number $\varepsilon$, this defines an $H$-polar cylinder since $$S\setminus (C \cup \tilde{T}\cup E_1\cup\ldots\cup E_7)\cong \mathbb{P}^2\setminus (\pi(C)\cup T)$$  (see Example~\ref{example:cusp}).
\end{proof}

For the following theorem we assume that $a_1\geqslant\ldots\geqslant a_r$ in \eqref{equation:coefficient}.

\begin{theorem}
\label{theorem:BR2} Suppose that  $H$ is of type $B(2)$. If one of the following conditions holds
\begin{itemize}
\item $2a_2>1$;
\item $2a_1+a_2>2$,
\end{itemize}
then $S$ contains an $H$-polar cylinder.
\end{theorem}
\begin{proof}
There are five $(-1)$-curves $E_3, \ldots, E_7$ on $S$ such that they, together with $E_1$ and $E_2$, define a birational morphism $\sigma:S\to\mathbb{P}^2$. Denote the point $\sigma(E_i)$ by $P_i$ for $i=1,\ldots, 7$.
Let $C_1$ be the conic that passes through the points $P_3,\ldots, P_7$.

 Suppose that the inequality $2a_2>1$ is satisfied.

 There is a conic $C_2$ passing through the points $P_1, P_2$ and meeting the conic $C_1$ only at a single point.
Let $T$ be the tangent line to both the conics $C_1$ and $C_2$ at the intersection point of $C_1$ and $C_2$.
For any real number $\varepsilon$ we have
$-K_{\mathbb{P}^2}\equiv \left(1+\varepsilon\right) C_1 +
\left(\frac{1}{2}-2\varepsilon\right)C_2+ 2\varepsilon T$. Hence,
\[
\begin{split}
-K_{S}&\sim_{\mathbb{Q}}\sigma^*\left(-K_{\mathbb{P}^2}\right)-\sum_{i=1}^{7} E_i\\
&\equiv\left(1+\varepsilon\right) \tilde{C}_1 +
\left(\frac{1}{2}-2\varepsilon\right)\tilde{C}_2+2\varepsilon \tilde{T}-\left(\frac{1}{2}+2\varepsilon\right)\left(E_1+E_2\right)+ \varepsilon\sum_{i=3}^{7}E_i,\\
\end{split}
\]
where $\tilde{C}_1$, $\tilde{C}_2$, $\tilde{T}$  are the proper transforms of $C_1$, $C_2$, and $T$, respectively.
Thus, we have
\[
\begin{split}
 H
&\equiv\frac{1}{\mu}\left\{\left(1+\varepsilon\right) \tilde{C}_1 +
\left(\frac{1}{2}-2\varepsilon\right)\tilde{C}_2+2\varepsilon \tilde{T}\right.+\\ 
&\phantom{\sim_{\mathbb{Q}}}+\left. \left(a_1-\frac{1}{2}-2\varepsilon\right)E_1+\left(a_2-\frac{1}{2}-2\varepsilon\right)E_2+ \varepsilon\sum_{i=3}^{7}E_i.\right\}\\
\end{split}
\]
Since $a_1-\frac{1}{2}\geqslant a_2-\frac{1}{2}>0$, for a sufficiently small  positive real number $\varepsilon$  this defines
 an $H$-polar cylinder on $S$.

Suppose that the inequality $2a_1+a_2>2$ is satisfied.

Let $L$ be a line passing through the point $P_2$ and tangent to the conic $C_1$. Let $C_3$ be the conic  that intersects $C_1$ only at the point where $C_1$ and $L$ meet and that passes through $P_1$.
For any real numbers $\beta$ and  $\varepsilon$ we have
$$-K_{\mathbb{P}^2}\equiv (1+2\varepsilon ) C_1 +
(\beta-\varepsilon)C_3+(1-2\beta-2\varepsilon)L.$$ Hence,
\[
\begin{split}
-K_{S}&\sim_{\mathbb{Q}}\sigma^*\left(-K_{\mathbb{P}^2}\right)-\sum_{i=1}^{7} E_i\\
&\equiv\left(1+2\varepsilon \right) \tilde{C}_1 +
\left(\beta-\varepsilon\right)\tilde{C}_3+\left(1-2\beta-2\varepsilon\right)\tilde{L}+\left(\beta-\varepsilon-1\right)E_1-2\left(\beta+\varepsilon\right)E_2+ 2\varepsilon\sum_{i=3}^{7}E_i,\\
\end{split}
\]
where  $\tilde{C}_3$, $\tilde{L}$ are the proper transforms of  $C_3$, $L$, respectively.
Thus, we have
\[
\begin{split}
 H
&\equiv\frac{1}{\mu}\left\{\left(1+2\varepsilon\right) \tilde{C}_1 +
\left(\beta-\varepsilon\right)\tilde{C}_3+\left(1-2\beta-2\varepsilon\right)\tilde{L}+\right. \\ &\phantom{\sim_{\mathbb{Q}}}+\left. \left(a_1+\beta-\varepsilon-1\right)E_1+\left(a_2-2\beta-2\varepsilon\right)E_2+ 2\varepsilon\sum_{i=3}^{7}E_i\right\}.\\
\end{split}
\]
By putting $\beta=\frac{3}{2}\varepsilon+1-a_1$  with a sufficiently small  positive real number $\varepsilon$, we are able to obtain
 an $H$-polar cylinder on $S$.
\end{proof}

From now on, we suppose that the ample $\mathbb{R}$-divisor $H$ is of type $C(7)$ with length $\ell$.
Then the morphism $\phi:S\to Z$ is a conic bundle, i.e., $Z=\mathbb{P}^1$. 
We may write
$$
K_{S}+\mu H\equiv aB+\sum_{i=1}^{\ell}a_iE_i
$$
where $a$ and $a_i$ are positive real numbers, $B$ is an irreducible fiber of $\phi$, and  $E_i$'s are disjoint $(-1)$-curves in fibers of $\phi$.  
There exist $(6-\ell)$ disjoint $(-1)$-curves $\hat{E}_1,\ldots, \hat{E}_{6-\ell}$ such that they are in fibers of $\phi$ and they generate the face $\Delta$ together with $B$ and $E_i$'s. Let $\phi_1: S\to W$ be the birational morphism obtained by contracting the disjoint $(-1)$-curves $E_1,\ldots, E_\ell$, $\hat{E}_1,\ldots, \hat{E}_{6-\ell}$.
Let $\hat{E}_j'$ be the $(-1)$-curve such that $\hat{E}_j+\hat{E}_j'$ is a fiber of $\phi$.

\begin{lemma}\label{theorem:no-cuspidal}
Suppose that the surface $S$ has no cuspidal rational curve in $|-K_S|$.
If the ample $\mathbb{R}$-divisor $H$ is of type $C(7)$ with length $3\leqslant \ell\leqslant 6$, then $S$ contains an $H$-polar cylinder.
\end{lemma}
\begin{proof}
The surface $S$ has exactly  twelve pairs of $(-1)$-curves $\{C_i, C_i'\}$, $i=1,\ldots, 12$ such that each $C_i+C'_i$ is a tacnodal anticanonical divisor. Choose one tacnodal anticanonical divisor, say $C_1+C_1'$.  Since we have more than ten tacnodal anticanonical divisors, we may assume that 
\begin{itemize}
\item none of $E_1,\ldots, E_\ell$ are  $C_1$ or $C_1'$; 
\item none of $\hat{E}_1,\hat{E}_1'\ldots, \hat{E}_{6-\ell}, \hat{E}_{6-\ell}'$ are $C_1$ or $C_1'$;
\item neither $\phi_1(C_1)$ nor $\phi_1(C_1')$ is a $(-1)$-curve on $W$.
\end{itemize}

Each of $E_i$ and $\hat{E}_j$ intersects  exclusively either $C_1$ or $C_1'$ once. Note that if $\hat{E}_i$ intersects $C_1$, then $\hat{E}_i'$ meets $C_1'$ and  if $\hat{E}_i$ intersects $C_1'$, then $\hat{E}_i'$ meets $C_1$.
Let $m_1$ (resp. $m_1'$) be the number of $E_i$ with $E_i\cdot C_1=1$ (resp. $E_i\cdot C_1'=1$) and let $m_0$ (resp. $m_0'$) be the number of $\hat{E}_j$ with  $\hat{E}_j\cdot C_1=1$ (resp. $\hat{E}_j\cdot C_1'=1$).  Put $m=m_0+m_1$ and $m'=m_0'+m_1'$. 
We may assume that 
\begin{itemize}
\item $m\geqslant m'$;
\item if $m=m'$, then $m_1\geqslant m_1'$;
\item $E_1,\ldots, E_{m_1}$ intersect $C_1$ and $E_{m_1+1},\ldots, E_{\ell}$ intersect $C_1'$;
\item $\hat{E}_1,\ldots, \hat{E}_{m_0}$ intersect $C_1$ and $\hat{E}_{m_0+1},\ldots, \hat{E}_{6-\ell}$ intersect $C_1'$.
\end{itemize} Since $m+m'=6$ and neither $\phi_1(C_1)$ nor $\phi_1(C_1')$ is a $(-1)$-curve on $W$, we have three possibilities, $(5,1)$, $(4,2)$, $(3,3)$ for $(m, m')$.

Suppose that $(m, m')=(5,1)$ and $W\cong\mathbb{P}^1\times\mathbb{P}^1$.
Then $\phi_1(C_1)$ is a $4$-curve. We may assume that this is an irreducible curve of bidegree $(1,2)$. Then $\phi_1(C_1')$ is a curve of bidegree $(1,0)$. The curves $\phi_1(C_1)$ and $\phi_1(C_1')$ meets tangentially at a single point $Q$. Let $L$ be the curve of bidegree $(0,1)$ passing through $Q$. Note that $\phi_1(B)$ is a curve of bidegree $(0,1)$ since $B\cdot C_1=B\cdot C_1'=1$.
We immediately see from Example~\ref{example:non-Prokhorov example8-1} that for an arbitrary real number $\varepsilon$
\begin{equation}\label{equation:5-1}
-K_{S}+aB\equiv(1-\varepsilon)C_1+(1+\varepsilon )C_1'+ (a+2\varepsilon) \tilde{L}-\sum^{m_1}_{i=1}\varepsilon E_i-\sum^{m_0}_{i=1}\varepsilon\hat{E}_i+\sum^{\ell}_{i=m_1+1}\varepsilon E_i+\sum^{6-\ell}_{i=m_0+1}\varepsilon\hat{E}_i,
\end{equation}
where $\tilde{L}$ is the proper transform of $L$ by $\phi_1$. In a similar way, we obtain 
\begin{equation}\label{equation:5-2}
-K_{S}+aB\equiv(1+\varepsilon)C_1+(1-\varepsilon )C_1'+ (a-2\varepsilon) \tilde{L}+\sum^{m_1}_{i=1}\varepsilon E_i+\sum^{m_0}_{i=1}\varepsilon\hat{E}_i-\sum^{\ell}_{i=m_1+1}\varepsilon E_i-\sum^{6-\ell}_{i=m_0+1}\varepsilon\hat{E}_i.
\end{equation}
Example~\ref{example:non-Prokhorov example8-1} also shows that the complements of the supports of the right hand sides of \eqref{equation:5-1} and \eqref{equation:5-2} are cylinders
for a sufficiently small positive real number $\varepsilon$.

Suppose that $(m, m')=(4,2)$. Then $\phi_{1}(C_1)$ is a $3$-curve, and hence $W\cong\mathbb{F}_1$. There is a unique $0$-curve
$M$ on $S$ such that  $\phi_1(M)$ is the $0$-curve passing through the intersection point of $\phi_1(C_1)$ and  $\phi_1(C_1')$.
From Example~\ref{example:non-Prokhorov example8-3} we obtain 
\begin{equation}\label{equation:4-1}
-K_{S}+aB\equiv(1-\varepsilon)C_1+(1+\varepsilon )C_1'+ (a+\varepsilon) M -\sum^{m_1}_{i=1}\varepsilon E_i-\sum^{m_0}_{i=1}\varepsilon\hat{E}_i+\sum^{\ell}_{i=m_1+1}\varepsilon E_i+\sum^{6-\ell}_{i=m_0+1}\varepsilon\hat{E}_i\end{equation}
for an arbitrary real number $\varepsilon$. We can also obtain 
\begin{equation}\label{equation:4-2}
-K_{S}+aB\equiv(1+\varepsilon)C_1+(1-\varepsilon )C_1'+ (a-\varepsilon) M+\sum^{m_1}_{i=1}\varepsilon E_i+\sum^{m_0}_{i=1}\varepsilon\hat{E}_i-\sum^{\ell}_{i=m_1+1}\varepsilon E_i-\sum^{6-\ell}_{i=m_0+1}\varepsilon\hat{E}_i.\end{equation}
With a sufficiently small positive real number $\varepsilon$, these two divisors on the right hand sides define cylinders  (see Example~\ref{example:non-Prokhorov example8-3}).
 
Suppose that $(m, m')=(3,3)$. Then $\phi_{1}(C_1)$ and $\phi_{1}(C_1')$  are $2$-curves, and hence $W\cong\mathbb{P}^1\times\mathbb{P}^1$. Moreover, $\phi_1(C_1)$ and $\phi_1(C_1')$ are irreducible curves of bidegree $(1,1)$ intersecting tangentially at a single point $Q$. Let $L_1$ and $L_2$ be the curves of bidegrees $(1,0)$ and $(0,1)$, respectively, passing through $Q$.
The curve $\phi_1(B)$ is a curve of bidegree $(1,0)$ or $(0,1)$. Without loss of generality, we may assume that $\phi_1(B)$ is a curve of bidegree $(1,0)$.
Then for an arbitrary real number $\varepsilon$
\begin{equation}\label{equation:3-1}
-K_{S}+aB\equiv(1-2\varepsilon)C_1+(1+\varepsilon )C_1'+ (a+\varepsilon) \tilde{L}_1+ \varepsilon \tilde{L}_2-\sum^{m_1}_{i=1}2\varepsilon E_i-\sum^{m_0}_{i=1}2\varepsilon\hat{E}_i+\sum^{\ell}_{i=m_1+1}\varepsilon E_i+\sum^{6-\ell}_{i=m_0+1}\varepsilon\hat{E}_i,
\end{equation}
where $\tilde{L}_1$ and $\tilde{L}_2$ are the proper transforms of $L_1$ and $L_2$ by $\phi_1$.
In particular, the complement of the support of  the divisor on the right hand side is a cylinder
for a sufficiently small positive real number $\varepsilon$
(see Example~\ref{example:non-Prokhorov example8-0}).

Now we construct $H$-polar cylinders case by case, according to the length $\ell$.

\bigskip
\textbf{Case 1.}  $\ell=6$ and $W\cong \mathbb{P}^1\times\mathbb{P}^1$.
\medskip

If $m=5$, then we use \eqref{equation:5-1} to obtain
\[\mu H\equiv(1-\varepsilon)C_1+(1+\varepsilon )C_1'+ (a+2\varepsilon) \tilde{L}+(a_6+\varepsilon) E_6+
\sum^5_{i=1}(a_i-\varepsilon)E_i.\]
If $m=3$, then we apply \eqref{equation:3-1} to yield
 \[\mu H\equiv(1-2\varepsilon)C_1+(1+\varepsilon )C_1'+ (a+\varepsilon) \tilde{L}_1+ \varepsilon \tilde{L}_2+\sum^3_{i=1} (a_i-2\varepsilon)E_i +\sum^6_{i=4}(a_i+\varepsilon)E_i.\]
These show that $S$ has an $H$-polar cylinder.

\bigskip
\textbf{Case 2.}  $\ell=6$ and $W\cong \mathbb{F}^1$.
\medskip

Suppose that $m=5$. Then $\phi_1(C_1)$ is a $4$-curve and  $\phi_1(C_1')$ is a $0$-curve.  They meet tangentially at a single point. There is a $(-1)$-curve $E$ on $S$ such that $\phi_1(E)$ is the negative section of $W\cong \mathbb{F}^1$. The curve $\phi_1(B)$ is equivalent to $\phi_1(C_1')$.
Therefore, 
\[-K_{S}+aB\equiv(1-\varepsilon)C_1+(1+a+2\varepsilon )C_1'+ 2\varepsilon E +(a+2\varepsilon) E_6-\varepsilon\sum^5_{i=1} E_i,\]
and 
\[\mu H\equiv(1-\varepsilon)C_1+(1+a+2\varepsilon )C_1'+ 2\varepsilon E +(a_6+a+2\varepsilon) E_6+\sum^5_{i=1}(a_i-\varepsilon) E_i.\]

Suppose that $m=4$. Then we apply \eqref{equation:4-1} to obtain \[\mu H\equiv(1-\varepsilon)C_1+(1+\varepsilon )C_1'+ (a+\varepsilon) M +\sum^4_{i=1}(a_i-\varepsilon) E_i +\sum^6_{i=4}(a_i+\varepsilon)E_i.\]

The divisors on the right hand sides produce $H$-polar cylinders on $S$.

\bigskip

From now, we consider the cases where $\ell<6$. By contracting $\hat{E}_1'$ instead of $\hat{E}_1$, if necessary,
we may always assume that $W\cong \mathbb{P}^1\times\mathbb{P}^1$.

\bigskip
\textbf{Case 3.}  $\ell=5$.
\medskip

If $(m_1,m_0)=(5,
0)$, we apply  \eqref{equation:5-1} to yield  
\[\mu H\equiv(1-\varepsilon)C_1+(1+\varepsilon )C_1'+ (a+2\varepsilon) \tilde{L}+\varepsilon \hat{E}_1+
\sum^5_{i=1}(a_i-\varepsilon)E_i. \]
If $(m_1,m_0)=(4,1)$, then we use \eqref{equation:5-2} to obtain 
\[\mu H\equiv(1+\varepsilon)C_1+(1-\varepsilon )C_1'+ (a-2\varepsilon) \tilde{L}+\varepsilon \hat{E}_1+(a_5-\varepsilon) E_5+
\sum^4_{i=1}(a_i+\varepsilon)E_i. \]
If $(m_1,m_0)=(3,0)$, then \eqref{equation:3-1} shows
\[\mu H \equiv(1-2\varepsilon)C_1+(1+\varepsilon )C_1'+ (a+\varepsilon) \tilde{L}_1+ \varepsilon \tilde{L}_2+\varepsilon\hat{E}_1+\sum^{3}_{i=1}(a_i-2\varepsilon) E_i+\sum^{5}_{i=4}(a_i+\varepsilon) E_i.\]
In each case, the divisor on the right hand side produces an $H$-polar cylinder on $S$ 
with a sufficiently small positive real number $\varepsilon$.

\bigskip
\textbf{Case 4.}  $\ell=4$.
\medskip

Suppose that $(m_1, m_0)=(4,1)$. When we obtain the birational morphism $\phi_1$, we contract $\hat{E}_1'$ instead of $\hat{E}_1$. Then this new contraction maps $S$ onto $\mathbb{F}_1$. The curve $C_1$ meets $E_1,\ldots, E_4$ and the curve $C_1'$ intersects $\hat{E}_1'$ and $\hat{E}_2$. We apply \eqref{equation:4-1} to this new set-up. Then we obtain 
\[\mu H\equiv(1-\varepsilon)C_1+(1+\varepsilon )C_1'+ (a+\varepsilon) M +\varepsilon \hat{E}_1'+\varepsilon \hat{E}_2 +\sum^4_{i=1}(a_i-\varepsilon) E_i.\]

If $(m_1, m_0)=(3,2)$, then we apply \eqref{equation:5-2} to yield  
\[\mu H\equiv(1+\varepsilon)C_1+(1-\varepsilon )C_1'+ (a-2\varepsilon) \tilde{L}+(a_4-\varepsilon) E_4+\sum^2_{i=1}\varepsilon \hat{E}_i+
\sum^3_{i=1}(a_i+\varepsilon)E_i. \]

If $(m_1, m_0)=(3,0)$,  then use \eqref{equation:3-1}, and obtain
\[\mu H\equiv(1-2\varepsilon)C_1+(1+\varepsilon )C_1'+ (a+\varepsilon) \tilde{L}_1+ \varepsilon \tilde{L}_2+(a_4+\varepsilon ) E_4+\sum^{3}_{i=1}(a_i-2\varepsilon ) E_i+\sum^{2}_{i=1}\varepsilon\hat{E}_i.\]

If $(m_1,m_0)=(2,1)$, then we contract $\hat{E}_2'$ instead of $\hat{E}_2$ when we obtain the birational morphism $\phi_1$. 
This new contraction sends $S$ to $\mathbb{F}_1$. The curve $C_1$ meets $E_1, E_2, \hat{E}_1, \hat{E}_2' $ and the curve $C_1'$ intersects $E_2$ and $E_4$. Apply \eqref{equation:4-2} to the new contraction, and we obtain 
\[\mu H\equiv(1+\varepsilon)C_1+(1-\varepsilon )C_1'+ (a-\varepsilon) M +\varepsilon \hat{E}_1+\varepsilon \hat{E}_2' +\sum^2_{i=1}(a_i+\varepsilon) E_i+\sum^4_{i=3}(a_i-\varepsilon) E_i.\]

These four equivalences show that $S$ has an $H$-polar cylinder if $\ell=4$.

\bigskip
\textbf{Case 5.}  $\ell=3$.
\medskip

If $(m_1, m_0)=(3,2)$, then we contract $\hat{E}_1'$ and $\hat{E}_2'$ instead of $\hat{E}_1$ and $\hat{E}_2$ when we obtain the birational morphism $\phi_1$. Then new contraction maps $S$ to $\mathbb{P}^1\times\mathbb{P}^1$. Therefore, it is enough to consider the case $(m_1, m_0)=(3,0)$ below.

If $(m_1,m_0)=(2,3)$, then apply \eqref{equation:5-2} and we obtain
\[\mu H\equiv(1+\varepsilon)C_1+(1-\varepsilon )C_1'+ (a-2\varepsilon) \tilde{L}+(a_3-\varepsilon) E_3+\sum^{2}_{i=1}(a_i+\varepsilon) E_i+\sum^{3}_{i=1}\varepsilon\hat{E}_i.
\]

If $(m_1,m_0)=(3,0)$, then we use \eqref{equation:3-1} to get 
\[\mu H\equiv(1-2\varepsilon)C_1+(1+\varepsilon )C_1'+ (a+\varepsilon) \tilde{L}_1+ \varepsilon \tilde{L}_2+\sum^{3}_{i=1}(a_i-2\varepsilon) E_i+\sum^{3}_{i=1}\varepsilon\hat{E}_i.
\]

Suppose that $(m_1,m_0)=(2,1)$. Then we contract $\hat{E}_2'$ and $\hat{E}_3'$ instead of $\hat{E}_2$ and $\hat{E}_3$. This new contraction reduces this case to the case where $(m_1,m_0)=(2,3)$ above. 

Consequently, these two equivalences verify  that $S$ has an $H$-polar cylinder.
\end{proof}
\begin{theorem}\label{theorem:CL3-6}
If the ample $\mathbb{R}$-divisor $H$ is of type $C(7)$ with length $3\leqslant \ell\leqslant 6$, then $S$ contains an $H$-polar cylinder.
\end{theorem}

\begin{proof}
Due to Lemma~\ref{theorem:no-cuspidal}, we may assume that there exists a cuspidal rational curve  $C$ in $|-K_S|$. Let $P$ be the point at which the curve $C$ has the cusp. Each $E_i$ intersects the curve $C$ at a single smooth point. 

We construct $H$-polar cylinders case by case, according to the length $\ell$.

\bigskip
\textbf{Case 1.}  $\ell=6$ and $W\cong \mathbb{P}^1\times\mathbb{P}^1$.
\medskip

There are   two $0$-curves $F_1, F_2$ passing through the point $P$ and not meeting any of $E_i$'s. 
Note that the curve $B$ must intersect one of the $0$-curves $F_1$, $F_2$. We may assume that it intersects $F_1$. Then $B\equiv F_2$.
We see that
\[-K_{S}\equiv (1-\varepsilon)C+2\varepsilon F_1+(a+2\varepsilon)F_2-aB-\varepsilon\sum^6_{i=1}E_i, \]
and hence 
\[\mu H\equiv
(1-\varepsilon)C+2\varepsilon F_1+(a+2\varepsilon)F_2+\sum^6_{i=1}(a_i-\varepsilon)E_i.\]

Let $\pi:\tilde{S}\to S$ be the blow up at the point $P$ and let $E$ be the exceptional curve of $\pi$. Then $\tilde{S}$ is a weak del Pezzo surface of degree $1$ and it has exactly one $(-2)$-curve, the proper transform $\tilde{C}$ of $C$. Denote
 the proper transforms on $\tilde{S}$ of the curves
$E_1,\ldots, E_6$, $F_1$, $F_2$ by $\tilde{E}_1,\ldots, \tilde{E}_6$, $\tilde{F}_1$, $\tilde{F}_2$. 
Contracting the $(-1)$-curves $\tilde{E}_1,\ldots, \tilde{E}_6$, $\tilde{F}_1$, $\tilde{F}_2$, 
we obtain a birational morphism $\psi:\tilde{S}\to \mathbb{P}^2$.

For a sufficiently small positive real number $\varepsilon$, the divisor above defines an $H$-polar cylinder since
$$S\setminus (C \cup F_1\cup F_2\cup E_1\cup\ldots\cup E_6)\cong \mathbb{P}^2\setminus (\psi(\tilde{C})\cup \psi(E)),$$ 
where $\psi(\tilde{C})$ and $\psi(E)$ are a conic and a line meeting tangentially at a single point on $\mathbb{P}^2$.

 \bigskip
 \textbf{Case 2.} $\ell=6$ and $W\cong \mathbb{F}_1$.  
 \medskip

 In this case, we have only one  $0$-curve $F$ passing through the point $P$ and not intersecting any of $E_i$'s. Instead, we consider the Zariski tangent line $M$ to $C$ at the point $P$. This is a $1$-curve. 
Note that the curve $\phi_1(B)$ is a $0$-curve and it intersects the unique $(-1)$-curve on $\mathbb{F}_1$.
We have
\[-K_{S}\equiv \left(1-\varepsilon\right)C+2\varepsilon M+(a+\varepsilon)F-aB-\varepsilon\sum^6_{i=1}E_i, \]
and hence 
\[\mu H\equiv
\left(1-\varepsilon\right)C+2\varepsilon M+(a+\varepsilon)F +\sum^6_{i=1}\left(a_i-\varepsilon\right)E_i.\]
For a sufficiently small positive real number $\varepsilon$, this defines an $H$-polar cylinder (see Example~\ref{example:non-Prokhorov example8-2}).

\bigskip
 \textbf{Case 3.} $\ell=5$.  
 \medskip
 
There are two $0$-curves $L_1$, $L_2$ such that \begin{itemize}
\item they pass through $P$;
\item they do not meet  each other outside $P$;
\item they are disjoint from   the curves $E_1,\ldots, E_5$.
\end{itemize}
 In addition, there is a unique $1$-curve 
$T$ that intersects $C$ only at $P$ and that does not meet any of $E_i$'s (see Example~\ref{example:non-Prokhorov example7-1}).
Note that the curve $B$ is a $0$-curve. It may be assumed to  intersect  $L_1$ but not $L_2$.
Then
\[-K_{S}\equiv (1-\varepsilon)C+\varepsilon L_1 +(a+\varepsilon)L_2 +\varepsilon T-aB-\varepsilon\sum^5_{i=1}E_i, \]
and hence 
\[\mu H\equiv
(1-\varepsilon)C+\varepsilon L_1 +(a+\varepsilon) L_2 +\varepsilon T+\sum^5_{i=1}(a_i-\varepsilon)E_i.\]
For a sufficiently small positive real number $\varepsilon$, this defines an $H$-polar cylinder
(see Example~\ref{example:non-Prokhorov example7-1}).

\bigskip
 \textbf{Case 4.} $\ell=4$.  
 \medskip

The surface $S$ has three  $0$-curves $F_1, F_2, F_3$ such that \begin{itemize}
\item they pass through $P$;
\item they do not intersect  each other outside $P$;
\item they are disjoint from   the curves $E_1$, $E_2$, $E_3$, $E_4$.
\end{itemize}
 In addition, it has two $1$-curves $G_1, G_2$
 such that \begin{itemize}
\item they meet $C$ only at $P$;
\item they do not intersect   each other outside $P$;
\item they are disjoint from   the curves $E_1$, $E_2$, $E_3$, $E_4$
\end{itemize}
(see Example~\ref{example:non-Prokhorov example6-1}).
The curve $B$ is a $0$-curve. We may assume that it intersects  $F_1$, $F_2$ but not $F_3$.
Then
\[-K_{S}\equiv(1-2\varepsilon)C+
\varepsilon(F_1+F_2)+(a+\varepsilon) F_3+
\varepsilon(G_1+G_2)-aB
-2\varepsilon\sum_{i=1}^4E_i. \]
Therefore, 
\[\mu H\equiv
(1-2\varepsilon)C+
\varepsilon(F_1+F_2)+(a+\varepsilon)F_3+
\varepsilon(G_1+G_2)+
\sum_{i=1}^4(a_i-2\varepsilon)E_i.\]
For a sufficiently small positive real number $\varepsilon$, this defines an $H$-polar cylinder as shown in  Example~\ref{example:non-Prokhorov example6-1}.

\bigskip
 \textbf{Case 5.} $\ell=3$.  
 \medskip

In this case, the surface $S$ has five  $0$-curves $F_1,\ldots, F_5$ such that
\begin{itemize}
\item they pass through $P$;
\item they do not intersect  each other outside $P$;
\item they are disjoint from   the curves $E_1$, $E_2$,  $E_3$
\end{itemize}
(see Example~\ref{example:non-Prokhorov example5-4}).
 The curve $B$ is a $0$-curve. We may assume that it meets  $F_1$, $F_2$, $F_3$, $F_4$ but not $F_5$.
Then \[-K_{S}\equiv
 (1-2\varepsilon)C+\varepsilon\sum_{i=1}^4F_i +(a+\varepsilon)F_5-aB-2\varepsilon(E_1+E_2+E_3).\]
Therefore, 
\[\mu H\equiv
 (1-2\varepsilon)C+\varepsilon\sum_{i=1}^4F_i +(a+\varepsilon)F_5+(a_1-2\varepsilon)E_1+(a_2-2\varepsilon)E_2+(a_3-2\varepsilon)E_3.\]
For a sufficiently small positive real number $\varepsilon$, this defines an $H$-polar cylinder.
 \end{proof}

\begin{theorem}\label{theorem:CL0-2}
If $H$ is of type $C(7)$ with  $a>\frac{10}{3}$, then  $S$ contains an $H$-polar cylinder.
\end{theorem}
\begin{proof}
Put $\phi_1(E_i)=P_i$  for $ i=1,\ldots, 6$.

Suppose
that $W\cong \mathbb{P}^1\times \mathbb{P}^1$. We may assume that $\phi_1(B)$ is a curve of bidegree $(0,1)$. Let $C$ be the curve of bidegree $(1,0)$ passing through the point $P_1$.
Let $F_i$ be the curve of bidegree $(0,1)$ passing through the point $P_i$ for $ i=1,\ldots, 6$.
We have
\[-K_{\mathbb{P}^1\times\mathbb{P}^1}\equiv 2C+\frac{1}{3}\sum_{i=1}^6F_i.\]
We then obtain
\[-K_{S}\equiv 2\tilde{C}+\frac{1}{3}\tilde{F}_1+\frac{4}{3}E_1+\frac{1}{3}\sum_{i=2}^6\tilde{F}_i-\frac{2}{3}\sum_{i=2}^6E_i,\]
where $\tilde{C}$ and $\tilde{F}_i$'s are the proper transforms of $C$ and $F_i$'s by $\phi_1$, respectively.
Since $B\equiv \tilde{F}_i+E_i$ for each $i$, we have
\[-K_{S}+aB\equiv 2\tilde{C}+\frac{1}{3}\tilde{F}_1+\frac{4}{3}E_1+\left(\frac{a}{5}+\frac{1}{3}\right)\sum_{i=2}^6\tilde{F}_i
+\left(\frac{a}{5}-\frac{2}{3}\right)\sum_{i=2}^6E_i.\]
Since $\frac{a}{5}-\frac{2}{3}>0$, Example~\ref{example:P1-bundle} shows that $S$ has an $H$-polar cylinder.

We now suppose that $W\cong\mathbb{F}_1$. Let $C$ be the negative section of $\mathbb{F}_1$. Note that $C$ cannot pass through any of the points $P_i$. Take the fiber  $F_i$ of the $\mathbb{P}^1$-bundle morphism of $\mathbb{F}_1$ to $\mathbb{P}^1$ that passes through the point $P_i$ for each $i$. We have
\[-K_{\mathbb{F}_1}\equiv 2C+\frac{1}{2}\sum_{i=1}^6F_i.\]
We then obtain
\[-K_{S}\equiv 2\tilde{C}+\frac{1}{2}\sum_{i=1}^6(\tilde{F}_i-E_i),\]
where $\tilde{C}$ and $\tilde{F}_i$'s are the proper transforms of $C$ and $F_i$ by $\phi_1$, respectively.
Since $B\equiv \tilde{F}_i+E_i$ for each $i$, we have
\[-K_{S}+aB\equiv 2\tilde{C}+\left(\frac{a}{6}+\frac{1}{2}\right)\sum_{i=1}^6\tilde{F}_i
+\left(\frac{a}{6}-\frac{1}{2}\right)\sum_{i=1}^6E_i.\]
Since $a>3$,  Example~\ref{example:lines} verifies that $S$ has an $H$-polar cylinder.
\end{proof}

Theorems~\ref{theorem:BR3-7}, \ref{theorem:BR2}, \ref{theorem:CL3-6} and  \ref{theorem:CL0-2} imply 
(1), (2), (3), and (4) in Theorem~ \ref{theorem:main-hard}, respectively.

\subsection{Cylinders in del Pezzo surfaces of degree $1$}
\label{section:cylinders-degree-1}

In order to prove Theorem~\ref{theorem:main-hard-1}, let
$S$ be a smooth del Pezzo surface of degree $1$ and let $H$ be an ample
$\mathbb{R}$-divisor on~$S$. 
We use the same notations as those at the beginning of Subsection~\ref{section:cylinders-degree-2}.

Again we first consider ample $\mathbb{R}$-divisors of type $B(r)$. Let $E_1,\ldots, E_r$ be the $r$ disjoint $(-1)$-curves that generate the face $\Delta$.
We may then write 
$$
K_{S}+\mu H\equiv \sum^{r}_{i=1}a_iE_i,
$$
for some positive real numbers $a_1, \ldots, a_r$. 
We may assume that $a_1\geqslant\ldots\geqslant a_r$.

\begin{proposition}
\label{proposition:deg-1-cylinders-birational} Suppose that
$r\geqslant 3$,  $2a_1+2a_2+a_3>4$, the
contraction $\phi$ is a birational, and
$Z\not\cong\mathbb{P}^1\times\mathbb{P}^1$. Then $S$ contains an
$H$-polar cylinder.
\end{proposition}

\begin{proof}

Note that $Z$ is a smooth del Pezzo surface of degree $r+1$.
Moreover, by our assumption
$Z\not\cong\mathbb{P}^1\times\mathbb{P}^1$. Thus, either
$Z=\mathbb{P}^2$ or $Z$ is a blow up of $\mathbb{P}^2$ at $(8-r)$
points in general position. For both the cases, let $\psi\colon Z\to\mathbb{P}^2$
be the blow up. If $8-r>0$, denote the
proper transforms of these $\psi$-exceptional curves on $S$ by
$E_{r+1},\ldots,E_8$. Put $P_i=\sigma(E_i)$ and
$\sigma=\psi\circ\phi$.

Let $C$ be the conic in $\mathbb{P}^2$ passing through the
points $P_4,\ldots,P_{8}$. Let $L$ be a line passing through the point $P_3$ and tangent to the conic $C$ and
let $Q$ be the intersection point of the line $L$ and the conic $C$. For $i=1,2$ let $C_i$ be the conic passing through the point $P_i$ and intersecting $C$ only at the point $Q$.

The open subset 
$U=\mathbb{P}^2\setminus (L\cup C\cup C_1\cup C_2)$ is a cylinder.

We claim that the cylinder $U^\prime:=\sigma^{-1}(U)\simeq U$ is
$H$-polar.
Indeed, for a real number  $\varepsilon $ we have
$$-K_{\mathbb{P}^2}\equiv
(1+3\varepsilon) C+\left(\alpha_1-\varepsilon\right) C_1+\left(\alpha_2-\varepsilon\right) C_2 
+\left(\alpha_3-2\varepsilon\right)L,$$ 
where $\alpha_1, \alpha_2, \alpha_3 >0$ and $2\alpha_1+2\alpha_2+\alpha_3=1$.
Hence,
\[\begin{split}
-K_{S}&\sim\sigma^*(-K_{\mathbb{P}^2})-\sum_{i=1}^{8}
E_i\\
&\equiv
(1+3\varepsilon) \tilde{C}+\left(\alpha_1-\varepsilon\right) \tilde{C}_1+\left(\alpha_2-\varepsilon\right) \tilde{C}_2 
+\left(\alpha_3-2\varepsilon\right)\tilde{L}+ \\ &+\left(\alpha_1-\varepsilon-1\right)E_1+\left(\alpha_2-\varepsilon-1\right)E_2+\left(\alpha_3-2\varepsilon-1\right)E_3+
3\varepsilon\sum_{i=4}^{8} E_i,\\
\end{split}
\]
where $\tilde{L}$, $\tilde{C}$ and $\tilde{C}_j$ are the proper transforms 
of the line $L$ and the conics  $C$, $C_j$, respectively. We then obtain
\[\begin{split}
\mu H &\equiv
(1+3\varepsilon) \tilde{C}+\left(\alpha_1-\varepsilon\right) \tilde{C}_1+\left(\alpha_2-\varepsilon\right) \tilde{C}_2 
+\left(\alpha_3-2\varepsilon\right)\tilde{L}+ \\ &+\left(\alpha_1+a_1-\varepsilon-1\right)E_1+\left(\alpha_2+a_2-\varepsilon-1\right)E_2+\left(\alpha_3+a_3-2\varepsilon-1\right)E_3+
\sum_{i=4}^{8}(3\varepsilon+a_i) E_i,\\
\end{split}
\]
where $a_i=0$ if $i>r$. 
Put $\alpha_1=\frac{3}{2}\varepsilon+1-a_1$, $\alpha_2=\frac{3}{2}\varepsilon+1-a_2$ and $\alpha_3=2a_1+2a_2-6\varepsilon-3$. Since $2a_1+2a_2+a_3>4$, for a sufficiently small positive real number $\varepsilon$, all the coefficients in the divisor above are positive. This proves our claim.
\end{proof}

We now suppose that the morphism $\phi:S\to Z$ is a conic bundle, i.e., $Z=\mathbb{P}^1$. 
We may write
$$
K_{S}+\mu H\equiv aB+\sum_{i=1}^{7}a_iE_i
$$
where $B$ is an irreducible fiber of $\phi$, $E_i$'s are disjoint $(-1)$-curves in fibers of $\phi$, $a$ is a positive real number, and $a_i$'s are non-negative real numbers. 

\begin{proposition}\label{proposition:deg-1-cylinders-conic}
If $H$ is  of type $C(8)$ with $a>\frac{30}{7} $,
then $S$ has $H$-polar cylinders. \end{proposition}
\begin{proof}
The proof of Theorem~\ref{theorem:CL0-2} works almost verbatim for this case.
\end{proof}

\bigskip
\bigskip

\textbf{Acknowledgements.} 
The authors are grateful to the referees for their invaluable comments.  In particular, they pointed out gaps in the original  proofs of Theorems~\ref{theorem:BR3-7} and~\ref{theorem:CL3-6}.  The gaps have been filled up in the revised version. Their  comments enable the authors to improve the results as well as the
exposition.
The first author was supported within the framework of a subsidy granted to the HSE
by the Government of the Russian Federation for the implementation of the Global Competitiveness Program.
The second author has been supported by IBS-R003-D1, Institute for Basic Science in Korea and the third author by NRF-2014R1A1A2056432, the National Research Foundation in Korea.

\end{document}